\documentclass[preprint]{imsart}
\RequirePackage[OT1]{fontenc}
\RequirePackage{amsthm,amsmath,amsfonts}
\RequirePackage[numbers]{natbib}
\RequirePackage[colorlinks,citecolor=blue,urlcolor=blue]{hyperref}
\usepackage{amsfonts}
\usepackage{amscd}
\usepackage{amssymb}
\usepackage{amsthm}
\usepackage{amsmath}
\usepackage{latexsym}
\newtheorem{theorem}{Theorem}

\newtheorem{condition}[theorem]{Condition}

\newtheorem{lemma}[theorem]{Lemma}

\newtheorem{remark}[theorem]{Remark}

\allowdisplaybreaks

\startlocaldefs
\numberwithin{equation}{section}
\theoremstyle{plain}

\endlocaldefs

\begin{document}

\begin{frontmatter}
\title{A second order time discretization of the solution of the non-linear filtering problem}
\runtitle{Second order discretization of the filtering problem}

\begin{aug}
\author{\fnms{Dan} \snm{Crisan}\thanksref{t1}\ead[label=e1]{d.crisan@imperial.ac.uk}},
\author{\fnms{Salvador} \snm{Ortiz-Latorre}\thanksref{t2}\ead[label=e2]{s.ortiz-latorre@cma.uio.com}}

\thankstext{t1}{The work of D. C. was partially supported by the EPSRC Grant EP/H0005500/1.}
\thankstext{t2}{The work of S. O-L. was supported by the BP-DGR 2009 grant and the project ÓEnergy Markets: Modeling, Optimization and Simulation (EMMOS)Ó, funded by the Norwegian Research Council under grant Evita/205328.}

\runauthor{D. Crisan and S. Ortiz-Latorre}

\affiliation{Imperial College London and Centre of Mathematics for Applications}

\address{Department of Mathematics\\
Imperial College London\\
Huxley's Building\\
180 Queen's Gate\\
London SW7 2AZ\\
United Kingdom\\
\printead{e1}\\
}

\address{Centre of Mathematics for Applications\\
Oslo University\\
P.O. Box 1053 Blindern\\
N-0316 Oslo\\
Norway\\
\printead{e2}\\
}
\end{aug}

\begin{abstract}
The solution of the continuous time filtering problem can be represented as a
ratio of two expectations of certain functionals of the signal process that
are parametrized by the observation path. We introduce a new time
discretisation of these functionals corresponding to a chosen partition of the
time interval and show that the convergence rate of discretisation is
proportional with the square of the mesh of the partition.
\end{abstract}

\begin{keyword}[class=MSC]
\kwd[Primary ]{60F05,60F25,60G35,60H35,93E11}
\end{keyword}

\begin{keyword}
\kwd{Nonlinear filtering}
\kwd{second order time discretisations}
\kwd{Kallianpur-Striebel's formula}
\end{keyword}

\end{frontmatter}

\section{Introduction}

Partially observed dynamical systems are ubiquitous in a multitude of
real-life phenomena. The dynamical system is typically modelled by a
continuous time stochastic process called the signal process $X$. The signal
process cannot be measured directly, but only via a related process $Y$,
called the observation process. The filtering problem is that of estimating
the current state of the dynamical system at the current time given the
observation data accumulated up to that time. Mathematically the problem
entails computing the conditional distribution of the signal process $X_{t}$,
denoted by $\pi_{t}$, given $\mathcal{Y}_{t},$ the $\sigma$-algebra generated
by $Y$. In a few special cases, $\pi_{t}$ can be expressed in closed form as a
functional of the observation path. For example, the celebrated Kalman-Bucy
filter does this in the linear case. In general, an explicit formula for
$\pi_{t}$ is not available and inferences can only be made by numerical
approximations of $\pi_{t}$. As expected the problem has attracted a lot of
attention in the last fifty years (see Chapter 8 of \cite{BaCr08} for a survey
of existing numerical methods for approximating $\pi_{t}$. Particle
methods\footnote{Also known as \emph{particle filters }or \emph{sequential
Monte Carlo methods}.} are algorithms which approximate $\pi_{t}$ with
discrete random measures of the form $\sum_{i}a_{i}(t)\delta_{v_{i}(t)},$ in
other words with empirical distributions associated with sets of randomly
located particles of stochastic masses $a_{1}(t)$,$a_{2}(t)$, \dots, which
have stochastic positions $v_{1}(t)$,$v_{2}(t)$, \dots\ . These methods are
currently among the most successful and versatile for numerically solving the
filtering problem. The basis of this class of numerical methods is the
representation of $\pi_{t}$ given by the Kallianpur--Striebel formula (see
(\ref{ks}) below). In the case when the signal process is modelled by the
solution of a stochastic differential equation (SDE) and the observation
process is a function of the signal perturbed by white noise (see Section 2
below for further details), the formula entails the computation of
expectations of functionals of the solution of the signal SDE that are
parametrized by the observation path. The numerical approximation of $\pi_{t}$
requires three procedures:

\begin{itemize}
\item the discretization of the functionals. The discretization corresponds to
a choice of a partition of the time interval $\left[  0,t\right]  .$

\item the approximation of the law of the signal with a discrete measure.

\item the control of the computational effort.
\end{itemize}

The first step is typically achieved by the discretization scheme introduced
by Picard in \cite{Pi84}. This offers a first order approximation for the
functionals appearing in formula (\ref{ks}). More precisely, the $L_{1}$-rate
of convergence of the approximation is proportional with the mesh of the
partition of the time interval $\left[  0,t\right]  $ (see Theorem 21.5 in
\cite{Cris11}). The second and the third step are achieved by a combination of
an Euler approximation of the solution of the SDE, a Monte Carlo step that
gives a sample from the law of the Euler approximation and a re-sampling step
that acts as a variance reduction method and keeps the computational effort in
control. There are a variety of algorithms that follow this template. Further
details can be found, for instance, in Part VII of \cite{CrRo11}. It is worth
pointing out that once the functional discretization and the Euler
approximation have been applied, the problem can be reduced to one where the
signal evolves and is observed in discrete time. The discrete version of the
filtering problem is popular both with practitioners and with theoreticians.
The majority of the existing theoretical results and the numerical algorithms
are constructed and analyzed in the discrete framework. For more details, the
interested reader can consult the comprehensive theoretical monograph
\cite{Delm04} and the reference therein and the equally comprehensive
methodological volume \cite{DFG01} and the references therein with some
updates in Part VII of \cite{CrRo11}.

The first order discretization introduced by Picard creates a bottleneck:
There exist higher order schemes for approximating the law of the signal that
can be used, but which won't bring any substantial improvements because of
this. For example, in the recent paper \cite{CrOr2013}, the authors employ
high order cubature methods to approximate the law of the signal with only
minimal improvements due to the low order discretization of the required
functionals. The aim of this paper is to address this issue. We introduce
below second order discretization of the functionals. As we shall see, we
prove that the $L_{p}$-rate of convergence of the approximation is
proportional with the square of the mesh of the partition of the time interval
$\left[  0,t\right]  $. For details, see Theorem \ref{Theo_Main_Filtering_2}
below. In a subsequent paper \cite{SOL14}, this discretization procedure is
employed to produce a second order particle filter. It is hoped that this
discretization will be used in conjunction with other high order
approximations of the law of the signal, in particular with cubature methods.
It is worth mentioning we are not aware of any other similar discretization
scheme and that, even though a class of schemes of any order would be
desirable we haven't been able to construct one.\footnote{To be more precise,
we can construct discretization schemes of any order, but not recursive ones.
That means that the discretization at time $t_{1},$ cannot be constructed by
starting with the discretization at time $t_{2}<t_{1}$ and adding to it the
part corresponding to $\left[  t_{2},t_{1}\right]  $. Instead we need to redo
the discretization for the entire interval $\left[  0,t\right]  $, which will
lead to a non-recursive particle filter. By contrast, the functional
discretization presented here, as well as the original Picard discretisation,
are recursive. See Remark \ref{remarkhighorder} for further details.}

The paper is organized as follows: In Section \ref{Sec_MainResult} we
introduce some basic definitions and state the main result of the paper,
Theorem \ref{Theo_Main_Filtering_2}. Section \ref{Sec_MainResult} is devoted
to prove a general discretization result, Theorem \ref{Theo_Main}, from which
we will deduce our main result. In Section \ref{Sec_ProofTheo}, we state some
technical lemmas needed to apply Theorem \ref{Theo_Main} and we give the proof
of Theorem \ref{Theo_Main_Filtering_2}. Finally, in Section
\ref{Sec_Technical} we give the proof of the technical lemmas introduced in
the previous section.

\section{The framework}

Let $(\Omega,\mathcal{F},P)$ be a probability space together with a filtration
$(\mathcal{F}_{t})_{t\geq0}$ which satisfies the usual conditions. On
$(\Omega,\mathcal{F},P)$ we consider a $d_{X}\times d_{Y}$-dimensional
partially observed system $(X,Y)$ satisfying
\begin{align*}
X_{t}  &  =X_{0}+\int_{0}^{t}f(X_{s})ds+\int_{0}^{t}\sigma(X_{s})dV_{s},\\
Y_{t}  &  =\int_{0}^{t}h(X_{s})ds+W_{t},
\end{align*}
where $V$ is a standard $\mathcal{F}_{t}$-adapted $d_{V}$-dimensional Brownian
motion and and $W$ is a a standard $\mathcal{F}_{t}$-adapted $d_{Y}%
$-dimensional Brownian motion, independent of each other. We also denote by
$\pi_{0}$ the law of $X_{0}.$ We assume that $f=(f^{i})_{i=1}^{d_{X}%
}:\mathbb{R}^{d_{X}}\rightarrow\mathbb{R}^{d_{X}}$ and $\sigma=(\sigma_{j}%
^{i})_{\substack{i=1,...d_{X},\\j=1,...,d_{V}}}:\mathbb{R}^{d_{X}}%
\rightarrow\mathbb{R}^{d_{X}\times d_{V}}$ are globally Lipschitz continuous.
In addition, we assume that $h=\left(  h^{i}\right)  _{i=1}^{d_{Y}}%
:\mathbb{R}^{d_{X}}\rightarrow\mathbb{R}^{d_{Y}}$ is measurable and has linear growth.

Let $\{\mathcal{Y}_{t}\}_{t\geq0}$ be the usual augmentation of the filtration
associated with the process $Y,$ that is, $\mathcal{Y}_{t}=\sigma\left(
Y_{s},s\in\lbrack0,t]\right)  \vee\mathcal{N},$ where $\mathcal{N}$ are all
the $P$-null sets of $(\Omega,\mathcal{F},P)$. We are interested in
determining $\pi_{t},$ the conditional law of the signal $X$ at time $t$ given
the information accumulated from observing $Y$ in the interval $[0,t].$ More
precisely, for any Borel measurable and bounded function $\varphi,$ we want to
compute $\pi_{t}\left(  \varphi\right)  =\mathbb{E}[\varphi\left(
X_{t}\right)  |\mathcal{Y}_{t}].$ By an application of Girsanov's theorem one
can construct a new probability measure $\tilde{P}$ absolutely continuous with
respect to $P$ under which $Y\ $becomes a Brownian motion independent of $X$
in the law of $X$ remains unchanged. Moreover the process $Z=\left(
Z_{t}\right)  _{t\geq0}$ given by
\begin{equation}
Z_{t}=\exp\left(  \sum_{i=1}^{d_{Y}}\int_{0}^{t}h^{i}(X_{s})dY_{s}^{i}%
-\frac{1}{2}\sum_{i=1}^{d_{Y}}\int_{0}^{t}h^{i}\left(  X_{s}\right)
^{2}ds\right)  ,\ t\geq0. \label{Equ_Likelihood}%
\end{equation}
is an $\mathcal{F}_{t}$-adapted martingale under $\tilde{P}$. $\ $Let
$\mathbb{\tilde{E}}$ be the expectation with respect to $\not P  $.~In the
following we will make use of the measure valued process $\rho=\left(
\rho_{t}\right)  _{t\geq0},$ defined by the formula $\rho_{t}\left(
\varphi\right)  =\mathbb{\tilde{E}}[\varphi\left(  X_{t}\right)
Z_{t}|\mathcal{Y}_{t}],$ for any bounded Borel measurable function $\varphi
$.\ The two processes $\pi$ and $\rho$ are connected through the
Kallianpur-Striebel's formula:
\begin{align}
\pi_{t}\left(  \varphi\right)   &  =\frac{\rho_{t}\left(  \varphi\right)
}{\rho_{t}\left(  \boldsymbol{1}\right)  }\nonumber\\
&  =\frac{\mathbb{\tilde{E}}\left[  \varphi\left(  X_{t}\right)  \exp\left.
\left(  \sum_{i=1}^{d_{Y}}\int_{0}^{t}h^{i}(X_{s})dY_{s}^{i}-\frac{1}{2}%
\sum_{i=1}^{d_{Y}}\int_{0}^{t}h^{i}\left(  X_{s}\right)  ^{2}ds\right)
\right\vert \mathcal{Y}_{t}\right]  }{\mathbb{\tilde{E}}\left[  \exp\left.
\left(  \sum_{i=1}^{d_{Y}}\int_{0}^{t}h^{i}(X_{s})dY_{s}^{i}-\frac{1}{2}%
\sum_{i=1}^{d_{Y}}\int_{0}^{t}h^{i}\left(  X_{s}\right)  ^{2}ds\right)
\right\vert \mathcal{Y}_{t}\right]  }, \label{ks}%
\end{align}
P-a.s., where $\boldsymbol{1}$ is the constant function $\boldsymbol{1}\left(
x\right)  =1,x\in\mathbb{R}^{d}.\ $As\ a result, $\rho$ is called the
unnormalised conditional distribution of the signal. For further details on
the filtering framework, see \cite{BaCr08}.

It follows from (\ref{ks}) that $\pi_{t}\left(  \varphi\right)  $ is a ratio
of two conditional expectations of functionals of the signal that depend on
the stochastic integrals with respect to the process $Y.$ Hence, a second
order discretization of $\pi_{t}$ relies on the second order approximation of
these two expectations. We achieve this in Theorem \ref{Theo_Main_Filtering_2} below.

\section{Main result\label{Sec_MainResult}}

We introduce first some useful notation and definitions. We denote by
$\mathcal{B}_{b}$ the space of bounded Borel-measurable functions and by
$C_{b}^{k}$ the space of continuously differentiable functions up to order
$k\in\mathbb{Z}_{+}$ with bounded derivatives of order greater or equal to
one. Moreover, we denote by $C_{P}^{k}$ the space of continuously
differentiable functions up to order $k\in\mathbb{Z}_{+}$ such that the
function and its derivatives have at most polynomial growth.

In the following, we will use the notation introduced in Section 5.4 in
Kloeden and Platen \cite{KlPl92}. More precisely, let $S$ be a subset of
$\mathbb{Z}_{+}$ and denote by $\mathcal{M}^{\ast}(S)$ the set of all
multi-indices with values in $S.$ In addition, denote by $\mathcal{M}%
(S)\triangleq\mathcal{M}^{\ast}(S)\cup\{\varnothing\}$. For $\alpha
=(\alpha_{1},...,\alpha_{k})\in\mathcal{M}(S)$ denote by $\left\vert
\alpha\right\vert \triangleq k$ the length of $\alpha$ ($\left\vert
\varnothing\right\vert =0$), $\alpha_{-}\triangleq(\alpha_{1},...,\alpha
_{k-1})$ and $_{-}\alpha\triangleq(\alpha_{2},...,\alpha_{k})$. Given two
multi-indices $\alpha,\beta\in\mathcal{M}(S)$ we denote its concatenation by
$\alpha\ast\beta$ .We shall also consider the hierarchical set $\mathcal{M}%
_{m}(S)$ and its associated remainder set $\mathcal{M}_{m}^{R}(S),$ that is,
\[
\mathcal{M}_{m}(S)\triangleq\{\alpha\in\mathcal{M}(S):\left\vert
\alpha\right\vert \leq m\}
\]
and%
\[
\mathcal{M}_{m}^{R}(S)\triangleq\{\alpha\in\mathcal{M}(S):\left\vert
\alpha\right\vert =m+1\}.
\]
We shall use the sets of multi-indices with values in the sets $S_{0}%
=\{0,1,...,d_{V}\}$ and $S_{1}=\{1,...,d_{V}\}$.

For $\alpha\in\mathcal{M}(S_{0}),$ denote by $I_{\alpha}(h)_{s,t}$ the
following It\^{o} iterated integral%
\[
I_{\alpha}(h(\cdot))_{s,t}=\left\{
\begin{array}
[c]{ccc}%
h(t) & \text{if} & \left\vert \alpha\right\vert =0\\
\int_{s}^{t}I_{\alpha_{-}}(h(\cdot))_{s,u}du & \text{if} & \left\vert
\alpha\right\vert \geq1\quad\text{and}\quad\alpha_{|\alpha|}=0\\
\int_{s}^{t}I_{\alpha_{-}}(h(\cdot))_{s,u}dV_{u}^{\alpha_{|\alpha|}} &
\text{if} & \left\vert \alpha\right\vert \geq1\quad\text{and}\quad
\alpha_{|\alpha|}\neq0
\end{array}
\right.  ,
\]
where $h=\{h(t)\}_{t\in\lbrack0,T]}$ is an adapted process (satisfying
appropriate integrability conditions) and $\alpha_{0}=\varnothing.$ For
$\alpha\in\mathcal{M}(S_{0}),$ with $\alpha=(\alpha_{1},...,\alpha_{k}),$ the
differential operators $L^{\alpha}$ is defined by%
\begin{align*}
L^{\alpha}g  &  =L^{\alpha_{1}}\circ L^{\alpha_{2}}\circ\cdots\circ
L^{\alpha_{k}}g,\\
L^{\varnothing}g  &  =g,
\end{align*}
where $L^{0},L^{r},r=1,...,d_{V}$ are the differential operators defined by
\begin{align*}
L^{0}g(x)  &  \triangleq\left\langle f(x),\nabla g(x)\right\rangle +\frac
{1}{2}\sum_{k,l=1}^{d_{X}}(\sigma\sigma^{T})_{l}^{k}(x)\frac{\partial^{2}%
g}{\partial x^{k}\partial x^{l}}(x)\\
&  =\sum_{k=1}^{d_{X}}f^{k}(x)\frac{\partial g}{\partial x^{k}}(x)+\frac{1}%
{2}\sum_{k,l=1}^{d_{X}}\sum_{r=1}^{d_{V}}\sigma_{r}^{k}(x)\sigma_{r}%
^{l}(x)\frac{\partial^{2}g}{\partial x^{k}\partial x^{l}}(x).\\
L^{r}g(x)  &  \triangleq\left\langle \sigma_{r}(x),\nabla g(x)\right\rangle
=\sum_{k=1}^{d_{X}}\sigma_{r}^{k}(x)\frac{\partial g}{\partial x^{k}}(x),\quad
r=1,...,d_{V},
\end{align*}
where $g:\mathbb{R}^{d_{X}}\rightarrow\mathbb{R}$ belongs to $C_{P}^{2}\left(
\mathbb{R}^{d_{X}};\mathbb{R}\right)  .$

Let $\tau\triangleq\{0=t_{0}<\cdots<t_{i}<\cdots<t_{n}=t\}$ be a partition of
$[0,t],\delta_{i}\triangleq t_{i}-t_{i-1},\delta\triangleq\max_{i=1,...,n}%
\delta_{i}$ and $\tau(s)$ is the largest element of the partition smaller than
or equal to $s$, i.e., $\tau(s)\triangleq t_{i-1},s\in\lbrack t_{i-1}%
,t_{i}),i=1,...,n.$ We denote by $\Pi(t)$ the set of all partitions of $[0,t]$
such that $\delta$ converges to zero when $n$ tends to infinity and by
$\Pi(t,\delta_{0})$ the set of all partitions of $[0,t]$ such that $\delta$
converges to zero when $n$ tends to infinity and $\delta<\delta_{0}$.

To simplify the notation, we will add an additional component to the Brownian
motion $Y.$ Let $Y^{0}$ be the process $Y_{s}^{0}=s,$ for all $s\geq0$ and
consider the $(d_{Y}+1)$-dimensional process $Y=(Y^{i})_{i=0}^{d_{Y}}.$ Then
the martingale $Z$ defined in (\ref{Equ_Likelihood}) \ can be written as
\[
Z_{t}=\exp\left(  \sum_{0=1}^{d_{Y}}\int_{0}^{t}h^{i}(X_{s})dY_{s}^{i}\right)
,\ t\geq0,
\]
where $h^{0}=-\frac{1}{2}\sum_{i=1}^{d_{Y}}\left(  h^{i}\right)  ^{2}\ \ $For
$\tau\in\Pi(t),$ consider the process $Z^{\tau,2}=(Z_{t}^{\tau,2})_{t\geq0}$
given by%

\begin{align}
Z_{t}^{\tau,2} &  =\prod_{j=0}^{n-1}\exp\left(  \sum_{i=0}^{d_{Y}}%
h^{i}(X_{t_{j}})\left(  Y_{t_{j+1}}-Y_{t_{j}}\right)  +L^{0}h^{i}(X_{t_{j}%
})\int_{t_{j}}^{t_{j+1}}(s-t_{j})dY_{s}^{i}\right.  \nonumber\\
&  \hspace{1in}\left.  +L^{r}h^{i}(X_{t_{j}})\int_{t_{j}}^{t_{j+1}}(V_{s}%
^{r}-V_{\tau(s)}^{r})dY_{s}^{i}\right)  \nonumber\\
&  =\exp\left(  \sum_{i=0}^{d_{Y}}\int_{0}^{t}\left(  h^{i}(X_{\tau(s)}%
)+L^{0}h^{i}(X_{\tau(s)})(s-\tau(s))+\sum_{r=1}^{d_{V}}L^{r}h^{i}(X_{\tau
(s)})(V_{s}^{r}-V_{\tau(s)}^{r})\right)  dY_{s}^{i}\right)  \label{discretez}%
\end{align}
In the following, we will use the standard notation $L^{p}(\Omega
,\mathcal{F},\tilde{P})$ for the space of $p$-integrable random variables
(with respect to $\tilde{P}$) and denote by $\left\vert \left\vert
\cdot\right\vert \right\vert _{p},$ the corresponding norm on $L^{p}%
(\Omega,\mathcal{F},\tilde{P})$, i.e.., for $\xi\in L^{p}(\Omega
,\mathcal{F},\tilde{P})$,\ $\left\vert \left\vert \xi\right\vert \right\vert
_{p}\triangleq\mathbb{\tilde{E}}[\left\vert \xi\right\vert ^{p}]^{1/p}$. For
any Borel measurable function $\varphi$ such that $\varphi\left(
X_{t}\right)  Z_{t}^{\tau,2}\in L^{1}(\Omega,\mathcal{F},\tilde{P})$ define
\begin{align*}
\rho_{t}^{\tau,2}\left(  \varphi\right)   &  \triangleq\mathbb{\tilde{E}%
}[\varphi\left(  X_{t}\right)  Z_{t}^{\tau,2}|\mathcal{Y}_{t}],\\
\pi_{t}^{\tau,2}\left(  \varphi\right)   &  \triangleq\rho_{t}^{\tau,2}\left(
\varphi\right)  /\rho_{t}^{\tau,2}\left(  \boldsymbol{1}\right)  .
\end{align*}
Our main result is the following:

\begin{theorem}
\label{Theo_Main_Filtering_2}Suppose that $f,\sigma\in\mathcal{B}_{b}\cap
C_{b}^{2},$ $h\in\mathcal{B}_{b}\cap C_{b}^{2}\cap C_{P}^{4}$ and that $X_{0}$
has moments of all orders. Then, for any $p\geq1\ $and $\varphi\in C_{P}^{2}$
there exists a constant $C=C\left(  t,p,\varphi\right)  $ independent of
$\tau\in\Pi(t,\delta_{0}),$ where%
\begin{equation}
\delta_{0}=\frac{1}{2p\left\Vert Lh\right\Vert _{\infty}\sqrt{d_{Y}d_{V}}},
\label{Equ_Delta_0}%
\end{equation}
such that%
\[
\left\vert \left\vert \rho_{t}\left(  \varphi\right)  -\rho_{t}^{\tau
,2}\left(  \varphi\right)  \right\vert \right\vert _{p}\leq C\delta^{2}.
\]
Moreover, if $\sup_{\tau\in\Pi(t,\delta_{0})}\left\vert \left\vert \pi
_{t}^{\tau,2}(\varphi)\right\vert \right\vert _{2p+\varepsilon}<\infty,$ for
some $\varepsilon>0,$ then%
\[
\left\vert \left\vert \pi_{t}\left(  \varphi\right)  -\pi_{t}^{\tau,2}\left(
\varphi\right)  \right\vert \right\vert _{p}\mathbb{E}[|\pi_{t}\left(
\varphi\right)  -\pi_{t}^{\tau,2}\left(  \varphi\right)  |^{p}]\leq
C\delta^{2p},
\]
where $C$ is another constant independent of $\tau\in\Pi(t,\delta_{0})$.
\end{theorem}

\begin{remark}
\label{remarkhighorder}

i. The functional discretization given in (\ref{discretez}) is recursive. More
precisely, if $\tau^{\prime}\in\Pi(t+s)$ is a partition that includes $t$ as
an intermediate point, for example $\tau^{\prime}\triangleq\{0=t_{0}%
<\cdots<t_{k}=t<t_{k+1}\cdots<t_{n}=t+s\}$ with $0<k<n$, then%
\begin{align}
Z_{t+s}^{\tau^{\prime},2}  &  =Z_{t}^{\tau,2}\prod_{k=i}^{n-1}\exp\left(
\sum_{i=0}^{d_{Y}}h^{i}(X_{t_{k}})\left(  Y_{t_{k+1}}-Y_{t_{k}}\right)
+L^{0}h^{i}(X_{t_{k}})\int_{t_{k}}^{t_{k+1}}(s-t_{k})dY_{s}^{i}\right.
\nonumber\\
&  \hspace{1in}\left.  +L^{r}h^{i}(X_{t_{k}})\int_{t_{k}}^{t_{k+1}}(V_{s}%
^{r}-V_{\tau(s)}^{r})dY_{s}^{i}\right)  .\nonumber
\end{align}
This property is essential for implementation purposes as at every
discretization time we only need to use the previous functional discretization
and the term corresponding to the next interval to obtain the new functional discretization.

ii. The second order discretization presented above is obtained by making use
of the first order It\^{o}-Taylor expansion of $h^{i}\left(  X_{s}\right)  $
$i=0,...,d_{Y}$.\ Of course one can generalize this in the following
straightforward manner. Let $\xi^{\tau,m}=\left(  \xi_{i}^{\tau,m}\right)
_{i=0}^{d_{Y}},m\in\mathbb{N}$ be the random vectors obtained by using an
$\left(  m-1\right)  $-order It\^{o}-Taylor expansion of $h^{i}\left(
X_{s}\right)  $, more precisely%
\begin{align*}
\xi_{i}^{\tau,m}  &  \triangleq\sum_{j=1}^{n}\sum_{\alpha\in\mathcal{M}%
_{m-1}(S_{0})}L^{\alpha}h^{i}(X_{t_{j-1}})\int_{t_{j-1}}^{t_{j}}I_{\alpha
}(\boldsymbol{1})_{t_{j-1},s}dY_{s}^{i}\\
&  =\sum_{\alpha\in\mathcal{M}_{m-1}(S_{0})}\int_{0}^{t}L^{\alpha}%
h^{i}(X_{\tau(s)})I_{\alpha}(\boldsymbol{1})_{\tau(s),s}dY_{s}^{i}.
\end{align*}
Using this notation we can write
\[
\rho_{t}^{\tau,2}\left(  \varphi\right)  =\mathbb{\tilde{E}}\left[  \left.
\varphi\left(  X_{t}\right)  \exp\left(  \sum_{i=0}^{d_{Y}}\xi_{i}^{\tau
,2}\right)  \right\vert \mathcal{Y}_{t}\right]  .
\]
As an immediate generalization, we could replace $\xi_{i}^{\tau,2}$ with
$\xi_{i}^{\tau,m}$ to obtain an m-order discretization of $\rho_{t}^{\tau
}\left(  \varphi\right)  $. Unfortunately this is not possible as $\xi
_{i}^{\tau,m}$ does not have finite exponential moments for $m\geq3.$

iii. A \emph{non-recursive} $m$-order functional discretization can be
constructed, as follows%
\[
\rho_{t}^{\tau,2}\left(  \varphi\right)  =\mathbb{\tilde{E}}\left[  \left.
\varphi\left(  X_{t}\right)  \exp\left(  \Psi^{m}\left(  \sum_{i=0}^{d_{Y}}%
\xi_{i}^{\tau,m}\right)  \right)  \right\vert \mathcal{Y}_{t}\right]  ,
\]
where $\Psi^{m}$ is a suitably chosen truncation function. This result is an
immediate Corollary of Theorem \ref{Theo_Main} below.
\end{remark}

\section{A general approximation result\label{Sec_GeneralApprox}}

We will not prove Theorem \ref{Theo_Main_Filtering_2} directly. Instead, we
will first show a more general approximation result and we will deduce Theorem
\ref{Theo_Main_Filtering_2} as a consequence. We start by introducing some
technical conditions and recalling some basic results on martingale representations.

\begin{condition}
[\textbf{S}$(m)$]All moments of $X_{0}$ are finite. The functions
$f=(f^{i})_{i=1}^{d_{X}}:\mathbb{R}^{d_{X}}\rightarrow\mathbb{R}^{d_{X}%
},\sigma=(\sigma_{j}^{i})_{\substack{i=1,...d_{X},\\j=1,...,d_{V}}%
}:\mathbb{R}^{d_{X}}\rightarrow\mathbb{R}^{d_{X}\times d_{V}}$ belong to
$C_{P}^{m}$ and are globally Lipschitz.
\end{condition}

Note that if condition \textbf{S}$(m)$ holds for some $m\in\mathbb{N},$ then
condition \textbf{S}$(n)$ holds for any $n\leq m$\textbf{. }

\begin{remark}
\label{RemarkM}Under condition \textbf{S}$(m)$, in particular if the
coefficients are globally Lipschitz and all moments of $X_{0}$ are finite, the
signal process $X$ has moments of all orders and for any $p>0,$ we have
\[
\mathbb{\tilde{E}}\left[  \sup_{s\in\lbrack0,t]}|X_{s}|^{p}\right]  <\infty.
\]

\end{remark}

Following the notation from the previous section, let $\xi=(\xi_{i}%
)_{i=0}^{d_{Y}}$ be the random vector with entries%
\[
\xi_{i}=\int_{0}^{t}g^{i}(X_{s})dY_{s}^{i},\quad i=0,...,d_{Y},
\]
where $g:\mathbb{R}^{d_{X}}\rightarrow\mathbb{R}^{d_{Y}+1}.$ For the remainder
of the section we assume that $g$ satisfies the following regularity assumption:

\begin{condition}
[G$(m)$]The function $g:\mathbb{R}^{d_{X}}\rightarrow\mathbb{R}^{d_{Y}+1}$ is
a $C_{P}^{2m}$ function.
\end{condition}

Next, let $\xi^{\tau,m}=\left(  \xi_{i}^{\tau,m}\right)  _{i=0}^{d_{Y}}%
,m\in\mathbb{N}$ be the random vectors with entries%
\begin{align*}
\xi_{i}^{\tau,m}  &  \triangleq\sum_{j=1}^{n}\sum_{\alpha\in\mathcal{M}%
_{m-1}(S_{0})}L^{\alpha}g^{i}(X_{t_{j-1}})\int_{t_{j-1}}^{t_{j}}I_{\alpha
}(\boldsymbol{1})_{t_{j-1},s}dY_{s}^{i}\\
&  =\sum_{\alpha\in\mathcal{M}_{m-1}(S_{0})}\int_{0}^{t}L^{\alpha}%
g^{i}(X_{\tau(s)})I_{\alpha}(\boldsymbol{1})_{\tau(s),s}dY_{s}^{i}.
\end{align*}
Let $\varphi:\mathbb{R}^{d_{X}}\rightarrow\mathbb{R}$ be a measurable function
and $\psi:\mathbb{R}^{d_{Y}+1}\rightarrow\mathbb{R}$ be a continuously
differentiable function. We are interested in finding high order bounds of the
following quantity, called henceforth the approximation error,%
\[
q(X_{t},\xi,\xi^{\tau,m})\triangleq\mathbb{\tilde{E}}[\varphi(X_{t})(\psi
(\xi)-\psi(\xi^{\tau,m}))|\mathcal{Y}_{t}],
\]
in terms of $\delta,$ the size of the partition $\tau.$ Note that, by the mean
value theorem we can write%
\begin{equation}
\varphi(X_{t})(\psi(\xi)-\psi(\xi^{\tau,m}))=\sum_{i=0}^{d_{Y}}\eta_{i}%
^{\tau,m}(\xi_{i}-\xi_{i}^{\tau,m}), \label{EquMeanValueTheo}%
\end{equation}
where $\eta^{\tau,m}=(\eta_{i}^{\tau,m})_{i=0}^{d_{Y}}$ is the random vector
with components%
\begin{equation}
\eta_{i}^{\tau,m}=\int_{0}^{1}\varphi(X_{t})\partial_{i}\psi(s\xi
+(1-s)\xi^{\tau,m})ds. \label{EquDefEta_i}%
\end{equation}
Then, it is natural to consider the following set of conditions parametrized
by $p\geq1,m\in\mathbb{N}$ and $\Pi$ a set of partitions:

\begin{condition}
[\textbf{L}$(p,m,\Pi)$]There exists $\varepsilon>0$ such that
\begin{equation}
\sup_{\tau\in\Pi}\sup_{s\in\lbrack0,1]}\mathbb{\tilde{E}}[|\varphi
(X_{t})\partial_{i}\psi(s\xi+(1-s)\xi^{\tau,m})|^{2p+\varepsilon}]<\infty,
\label{CondLP}%
\end{equation}
for all $i=0,...,d_{Y}.$
\end{condition}

\begin{remark}
Note that $\xi$ has moments of all orders. As the functions $L^{\alpha}%
g^{i}(x),i=0,...,d_{Y},\alpha\in\mathcal{M}_{m}(S_{0})$ have polynomial
growth, then $\xi^{\tau,m}$ also has moments of all orders. If the function
$\varphi$ and the partial derivatives of $\psi$ have, at most, polynomial
growth, then condition\textbf{ L}$(p,m,\Pi)$ is satisfied$.$
\end{remark}

Note, however, that in the filtering problem the function $\psi$ is the
exponential and the previous remark will not apply. A different approach will
be required in the case to show that condition\textbf{ L}$(p,m,\Pi)$ is satisfied.

\begin{theorem}
\label{Theo_IntegralRepresentation}Assume condition\textbf{ L}$(p,m,\Pi)$ is
satisfied. Then there exists $\varepsilon>0$ such that\textbf{\ }the random
variables $\eta_{i}^{\tau,m}\in L^{2p+\varepsilon}(\Omega,\mathcal{F}%
,\tilde{P}),i=0,...,d_{Y},$ and admit the following martingale representation%
\[
\eta_{i}^{\tau,m}=\hat{\eta}_{i}^{\tau,m}+\sum_{r_{1}=1}^{d_{V}}\int_{0}%
^{t}\psi_{i,r_{1}}^{\tau,m}\left(  s_{1}\right)  dV_{s_{1}}^{r_{1}},\quad
i=0,...,d_{Y},
\]
where $\hat{\eta}_{i}^{\tau,m}=\mathbb{\tilde{E}}[\eta_{i}^{\tau
,m}|\mathcal{F}_{0}\vee\mathcal{Y}_{t}],$ belong to $L^{2p+\varepsilon}%
(\Omega,\mathcal{F}_{0}\vee\mathcal{Y}_{t},\tilde{P}),i=0,...,d_{Y}.$
Moreover, $\psi_{i,r_{1}}^{\tau,m}=\{\psi_{i,r_{1}}^{\tau,m}\left(
s_{1}\right)  ,s_{1}\in\lbrack0,t]\}$ are progressively measurable
$\mathcal{F}_{s_{1}}\vee\mathcal{Y}_{t}$-adapted processes such that
\[
\mathbb{\tilde{E}}\left[  \left(  \int_{0}^{t}\psi_{i,r_{1}}^{\tau,m}\left(
s_{1}\right)  ^{2}ds_{1}\right)  ^{(2p+\varepsilon)/2}\right]  <\infty,
\]
for all $i=0,...,d_{Y}$ and $r_{1}=1,...,d_{V}$.
\end{theorem}

By iterating the integral representation in Theorem
\ref{Theo_IntegralRepresentation}, one can get the following result.

\begin{theorem}
[Stroock-Taylor]\label{Theo_Stroock-Taylor}Assume condition \textbf{L}%
$(p,m,\Pi)$\textbf{,} then for any $k\in\mathbb{N}$\textbf{\ }the random
variables $\eta_{i}^{\tau,m}$ admit the following integral representation%
\[
\eta_{i}^{\tau,m}=\sum_{\beta\in\mathcal{M}_{k}(S_{1})}\hat{I}_{\beta}%
(\hat{\psi}_{i,\beta}^{\tau,m}(\cdot))_{0,t}+\sum_{\beta\in\mathcal{M}_{k}%
^{R}(S_{1})}\hat{I}_{\beta}(\psi_{i,\beta}^{\tau,m}(\cdot))_{0,t},
\]
where the kernels $\hat{\psi}_{i,\beta}^{\tau,m}$ and $\psi_{i,\beta}^{\tau
,m}$ satisfy the following recursive relationship
\begin{align*}
\hat{\psi}_{i,\varnothing}^{\tau,m}  &  \triangleq\mathbb{\tilde{E}}\left[
\eta_{i}^{\tau,m}|\mathcal{F}_{0}\vee\mathcal{Y}_{t}\right]  ,\\
\hat{\psi}_{i,\beta}^{\tau,m}(s_{1},...,s_{\left\vert \beta\right\vert })  &
\triangleq\mathbb{\tilde{E}}\left[  \psi_{i,\beta}^{\tau,m}(\cdot
)|\mathcal{F}_{0}\vee\mathcal{Y}_{t}\right]  ,\quad1\leq\left\vert
\beta\right\vert \leq k
\end{align*}
and
\begin{align*}
\eta_{i}^{\tau,m}  &  =\hat{\psi}_{i,\varnothing}^{\tau,m}+\sum_{r_{1}%
=1}^{d_{V}}\int_{0}^{t}\psi_{i,\beta_{1}}^{\tau,m}\left(  s_{1}\right)
dV_{s_{1}}^{\beta_{1}},\\
\psi_{i,\beta_{-}}^{\tau,m}(s_{1},...,s_{\left\vert \beta\right\vert -1})  &
=\hat{\psi}_{i,\beta_{-}}^{\tau,m}(s_{1},...,s_{\left\vert \beta\right\vert
-1})\\
&  \quad+\sum_{\beta_{\left\vert \beta\right\vert }=1}^{d_{V}}\int
_{0}^{s_{\left\vert \beta\right\vert -1}}\psi_{i,\beta_{-}\ast(\beta
_{\left\vert \beta\right\vert })}^{\tau,m}(s_{1},...,s_{\left\vert
\beta\right\vert -1},s_{\left\vert \beta\right\vert })dV_{s_{\left\vert
\beta\right\vert }}^{\beta_{\left\vert \beta\right\vert }}.
\end{align*}
If $\eta_{i}^{\tau,m}$ is Malliavin differentiable up to order $k+1$ then
\[
\hat{\psi}_{i,\beta}^{\tau,m}(s_{1},...,s_{|\beta|})=\mathbb{\tilde{E}}\left[
D_{s_{1},...,s_{|\beta|}}^{\beta_{1},...,\beta_{|\beta|}}\eta_{i}^{\tau
,m}|\mathcal{F}_{0}\vee\mathcal{Y}_{t}\right]  ,\quad0\leq\left\vert
\beta\right\vert \leq k,
\]
and%
\[
\psi_{i,\beta}^{\tau,m}(s_{1},...,s_{k+1})=\mathbb{\tilde{E}}\left[
D_{s_{1},...,s_{k+1}}^{\beta_{1},...,\beta_{k+1}}\eta_{i}^{\tau,m}%
|\mathcal{F}_{s_{k+1}}\vee\mathcal{Y}_{t}\right]  ,\quad\left\vert
\beta\right\vert =k+1.
\]

\begin{remark}
For any $\beta\in\mathcal{M}_{k}(S_{1})$, $\left\vert \beta\right\vert \geq1,$
we have that $\hat{\psi}_{i,\beta}^{\tau,m}(s_{1},...,s_{|\beta|})$ is
$\mathcal{F}_{0}\vee\mathcal{Y}_{t}$ measurable kernel defined on the simplex
\[
\mathcal{S}_{|\beta|}(t)\triangleq\{(s_{1},...,s_{|\beta|})\in\lbrack
0,t]^{|\beta|}:0\leq s_{\left\vert \beta\right\vert }<\cdots<s_{1}\leq t\},
\]
and
\[
\hat{I}_{\beta}(\hat{\psi}_{i,\beta}^{\tau,m}(s_{1},...,s_{|\beta|}%
))_{0,t}\triangleq\int_{0}^{t}\int_{0}^{s_{1}}\cdots\int_{0}^{s_{\left\vert
\beta\right\vert -1}}\hat{\psi}_{i,\beta}^{\tau,m}(s_{1},...,s_{|\beta
|})dV_{s_{|\beta|}}^{\beta_{|\beta|}}\cdots dV_{s1}^{\beta_{1}}.
\]
If $\beta\in\mathcal{M}_{k}^{R}(S_{1})$ the kernel $\psi_{i,\beta}^{\tau
,m}(s_{1},...,s_{k+1})$ is a $\mathcal{F}_{s_{k+1}}\vee\mathcal{Y}_{t}$
adapted process and
\[
\hat{I}_{\beta}(\psi_{i,\beta}^{\tau,m}(s_{1},...,s_{k+1}))_{0,t}=\int_{0}%
^{t}\int_{0}^{s_{1}}\cdots\int_{0}^{s_{k}}\psi_{i,\beta}^{\tau,m}%
(s_{1},...,s_{k+1})dV_{s_{k+1}}^{\beta_{k+1}}\cdots dV_{s_{1}}^{\beta_{1}}.
\]

\end{remark}
\end{theorem}

We also consider the following set of conditions parametrized by $p\geq
1,m\in\mathbb{N}$ and $\Pi$ a set of partitions.

\begin{condition}
[\textbf{UK}$(p,m,\Pi)$]There exists $\varepsilon>0$ such that the kernels
$\{\hat{\psi}_{i,\beta}^{\tau,m}\left(  \cdot\right)  \}_{\beta\in
\mathcal{M}_{m-1}(S_{1})}$ and $\{\psi_{i,\beta}^{\tau,m}\left(  \cdot\right)
\}_{\beta\in\mathcal{M}_{m-1}^{R}(S_{1})},$ given in Theorem
\ref{Theo_Stroock-Taylor}, satisfy%
\[
\sup_{\tau\in\Pi}\sup_{0\leq s_{|\beta|}<\cdots<s_{1}\leq t}\mathbb{\tilde{E}%
}\left[  \left\vert \hat{\psi}_{i,\beta}^{\tau,m}\left(  s_{1},...,s_{|\beta
|}\right)  \right\vert ^{2p\vee(2+\varepsilon)}\right]  <\infty,
\]
and%
\[
\sup_{\tau\in\Pi}\sup_{0\leq s_{m}<\cdots<s_{1}\leq t}\mathbb{\tilde{E}%
}\left[  \left\vert \psi_{i,\beta}^{\tau,m}\left(  s_{1},...,s_{m}\right)
\right\vert ^{2p\vee(2+\varepsilon)}\right]  <\infty,
\]
for all $i=0,...,d_{Y}.$
\end{condition}

The following theorem is gives the general discretisation error that will
allow us to deduce Theorem \ref{Theo_Main_Filtering_2}.

\begin{theorem}
\label{Theo_Main}Assume that conditions \textbf{S}$(m)$\textbf{,G}%
$(m)$\textbf{,L}$(p,m,\Pi)$ and \textbf{UK}$(p,m,\Pi)$\ hold. Then, there
exists a constant $C=C(t)$ independent of the partition $\tau\in\Pi$ such
that
\[
\left\vert \left\vert q(X_{t},\xi,\xi^{\tau,m})\right\vert \right\vert
_{2p}\leq C\delta^{m}.
\]

\end{theorem}

\begin{proof}
We can write%
\begin{align}
\xi_{i}-\xi_{i}^{\tau,m}  &  =\int_{0}^{t}\{g^{i}(X_{s})-\sum_{\alpha
\in\mathcal{A}_{m-1}}\int_{0}^{t}L^{\alpha}g^{i}(X_{\tau(s)})I_{\alpha
}(\boldsymbol{1})_{\tau(s),s}\}dY_{s}^{i}\nonumber\\
&  \triangleq\int_{0}^{t}\Theta_{s}^{g^{i},\tau,m}dY_{s}^{i}.
\label{EquDeltaXis}%
\end{align}
Next, by the It\^{o}-Taylor expansion with hierarchical set $\mathcal{M}%
_{m-1}(S_{0})$, see Theorem 5.5.1 in Kloeden-Platen \cite{KlPl92}, we have
that%
\[
\Theta_{s}^{g^{i},\tau,m}=\sum_{\alpha\in\mathcal{M}_{m-1}^{R}(S_{0}%
)}I_{\alpha}(L^{\alpha}g^{i}(X_{\cdot}))_{\tau(s),s}.
\]
Assumptions \textbf{S}$(m)$ and \textbf{G}$(m)$ imply the polynomial growth of
$L^{\alpha}g^{i},\alpha\in\mathcal{M}_{m-1}^{R}(S_{0}),i=1,...,d_{Y}.$ In
addition, one gets that
\begin{align*}
&  \mathbb{\tilde{E}}\left[  \xi_{i}-\xi_{i}^{\tau,m}|\mathcal{F}_{0}%
\vee\mathcal{Y}_{t}\right] \\
&  =\mathbb{\tilde{E}}\left[  \int_{0}^{t}\int_{\tau(s)}^{s}\int_{\tau
(s)}^{s_{m-1}}\cdots\int_{\tau(s)}^{s_{1}}L^{\alpha_{0}^{m}}g^{i}(X_{s_{0}%
})ds_{0}\cdots ds_{m-1}dY_{s}^{i}|\mathcal{F}_{0}\vee\mathcal{Y}_{t}\right] \\
&  =\int_{0}^{t}\int_{\tau(s)}^{s}\int_{\tau(s)}^{s_{m-1}}\cdots\int_{\tau
(s)}^{s_{1}}\mathbb{\tilde{E}}[L^{\alpha_{0}^{m}}g^{i}(X_{s_{0}}%
)|\mathcal{F}_{0}\vee\mathcal{Y}_{t}]ds_{0}\cdots ds_{m-1}dY_{s}^{i},
\end{align*}
where $\alpha_{0}^{m}=\overset{m}{\overbrace{(0,...,0)}}$ and if $m=1$ then
the integral is just over the simplex $\{(s_{0},s):\tau(s)\leq s_{0}\leq
s,0\leq s\leq t\}.$ Note that
\[
\mathbb{\tilde{E}}\left[  L^{\alpha_{0}^{m}}g^{i}(X_{s_{0}})|\mathcal{F}%
_{0}\vee\mathcal{Y}_{t}\right]  =\mathbb{\tilde{E}}\left[  L^{\alpha_{0}^{m}%
}g^{i}(X_{s_{0}})|\mathcal{F}_{0}\right]  =P_{s_{0}}L^{\alpha_{0}^{m}}%
g^{i}(X_{0}),
\]
where $P_{t}g\left(  x\right)  \triangleq\mathbb{\tilde{E}}[g(X_{t}^{x})]$ is
the semigroup associated to the signal, that is, to the SDE
\[
X_{t}^{x}=x+\int_{0}^{t}f(X_{s}^{x})ds+\int_{0}^{t}\sigma\left(  X_{s}%
^{x}\right)  dV_{s}.
\]
Moreover, under the assumption \textbf{S}$(m)$ the following bound holds. For
any $p\geq2,$%
\[
\mathbb{\tilde{E}}\left[  \sup_{s\in\lbrack0,t]}|X_{s}^{x}|^{p}\right]  \leq
C(p,t)(1+|x|^{p}).
\]
Hence, as $|L^{\alpha_{0}^{m}}g^{i}(x)|\leq C(g^{i},f,\sigma)(1+|x|^{r})$ for
some $r\geq2$, we have that%
\begin{align*}
P_{s_{0}}L^{\alpha_{0}^{m}}g^{i}(X_{0})  &  =\mathbb{E}[L^{\alpha_{0}^{m}%
}g^{i}(X_{s_{0}}^{x})]|_{x=X_{0}}\\
&  \leq\mathbb{E}[C\left(  g^{i},f,\sigma\right)  (1+|X_{s_{0}}^{x}%
|^{r})]|_{x=X_{0}}\\
&  \leq C\left(  g^{i},f,\sigma\right)  \left(  1+\mathbb{E}[\sup_{0\leq
s_{0}\leq t}|X_{s_{0}}^{x}|^{r}]|_{x=X_{0}}\right) \\
&  \leq C(g^{i},f,\sigma,r,t)(1+|X_{0}|^{r}).
\end{align*}
Taking into account equation (\ref{EquDeltaXis}) and using Theorem
\ref{Theo_Stroock-Taylor} with $k=m-1$, we can write%
\begin{align*}
q(X_{t},\xi,\xi^{\tau,m})  &  =\sum_{i=0}^{d_{Y}}\mathbb{\tilde{E}}[\hat{\eta
}_{i}^{\tau,m}\left(  \xi_{i}-\xi_{i}^{\tau,m}\right)  |\mathcal{Y}_{t}]\\
&  +\sum_{i=0}^{d_{Y}}\sum_{\beta\in\mathcal{M}_{m-1}(S_{1})}\mathbb{\tilde
{E}}[\hat{I}_{\beta}(\hat{\psi}_{i,\beta}^{\tau,m}(\cdot))_{0,t}\left(
\xi_{i}-\xi_{i}^{\tau,m}\right)  |\mathcal{Y}_{t}]\\
&  +\sum_{i=0}^{d_{Y}}\sum_{\beta\in\mathcal{M}_{m-1}^{R}(S_{1})}%
\mathbb{\tilde{E}}[\hat{I}_{\beta}(\psi_{i,\beta}^{\tau,m}(\cdot
))_{0,t}\left(  \xi_{i}-\xi_{i}^{\tau,m}\right)  |\mathcal{Y}_{t}]\\
&  \triangleq\sum_{i=0}^{d_{Y}}A_{i,1}+A_{i,2}+A_{i,3}.
\end{align*}
To finish the proof we will show that $\mathbb{\tilde{E}}[A_{i,j}^{2}]\leq
C\delta^{2m},i=0,...,d_{Y},j=1,...,3.$ for some constant $C$ that does not
depend on the partition $\tau.$

\underline{\textbf{Term }$A_{i,1}$}$:$

For any $i=1,...,d_{Y}$ and $\varepsilon>0$ as in condition \textbf{L}%
$(p,m,\Pi)$, we have that%
\begin{align*}
&  \mathbb{\tilde{E}}\left[  A_{i,1}^{2p}\right]  \\
&  =\mathbb{\tilde{E}}\left[  \mathbb{\tilde{E}}\left[  \hat{\eta}_{i}%
^{\tau,m}\mathbb{\tilde{E}}\left[  \left(  \xi_{i}-\xi_{i}^{\tau,m}\right)
|\mathcal{F}_{0}\vee\mathcal{Y}_{t}\right]  |\mathcal{Y}_{t}\right]
^{2p}\right]  \\
&  \leq\mathbb{\tilde{E}}\left[  \left(  \hat{\eta}_{i}^{\tau,m}\right)
^{2p}\left(  \int_{0}^{t}\int_{\tau(s)}^{s}\int_{\tau(s)}^{s_{m-1}}\cdots
\int_{\tau(s)}^{s_{1}}\mathbb{\tilde{E}}[L^{\alpha_{0}^{m}}g^{i}(X_{s_{0}%
})|\mathcal{F}_{0}\vee\mathcal{Y}_{t}]ds_{0}\cdots ds_{m-1}dY_{s}^{i}\right)
^{2p}\right]  ,\\
&  \leq\mathbb{\tilde{E}}\left[  \left(  \hat{\eta}_{i}^{\tau,m}\right)
^{2p+\varepsilon}\right]  ^{\frac{2p}{2p+\varepsilon}}\\
&  \times\mathbb{\tilde{E}}\left[  \left(  \int_{0}^{t}\int_{\tau(s)}^{s}%
\int_{\tau(s)}^{s_{m-1}}\cdots\int_{\tau(s)}^{s_{1}}\mathbb{\tilde{E}%
}[L^{\alpha_{0}^{m}}g^{i}(X_{s_{0}})|\mathcal{F}_{0}\vee\mathcal{Y}_{t}%
]ds_{0}\cdots ds_{m-1}dY_{s}^{i}\right)  ^{2p\frac{2p+\varepsilon}%
{\varepsilon}}\right]  ^{\frac{\varepsilon}{2p+\varepsilon}}\\
&  \leq\left\Vert \hat{\eta}_{i}^{\tau,m}\right\Vert _{L^{2p+\varepsilon
}(\tilde{P})}^{2p}\\
&  \times\mathbb{\tilde{E}}\left[  \left(  \int_{0}^{t}\left(  \int_{\tau
(s)}^{s}\int_{\tau(s)}^{s_{m-1}}\cdots\int_{\tau(s)}^{s_{1}}C(1+|X_{0}%
|^{r})ds_{0}\cdots ds_{m-1}\right)  ^{2}ds\right)  ^{p\frac{2p+\varepsilon
}{\varepsilon}}\right]  ^{\frac{\varepsilon}{2p+\varepsilon}}\\
&  \leq\left\Vert \hat{\eta}_{i}^{\tau,m}\right\Vert _{L^{2p+\varepsilon
}(\tilde{P})}^{2p}\mathbb{\tilde{E}}\left[  \left(  \int_{0}^{t}C^{2}%
(1+|X_{0}|^{r})^{2}\delta^{2m}ds\right)  ^{p\frac{2p+\varepsilon}{\varepsilon
}}\right]  ^{\frac{\varepsilon}{2p+\varepsilon}}\\
&  \leq C^{2p}t^{p}\left\Vert \hat{\eta}_{i}^{\tau,m}\right\Vert
_{L^{p}(\tilde{P})}^{2}\mathbb{\tilde{E}}\left[  (1+|X_{0}|^{r})^{2p\frac
{2p+\varepsilon}{\varepsilon}}\right]  ^{\frac{\varepsilon}{2p+\varepsilon}%
}\delta^{2pm}.
\end{align*}

\underline{\textbf{Term }$A_{i,2}$}$:$

For any $i=1,...,d_{Y},\beta=(\beta_{1},...,\beta_{|\beta|})\in\mathcal{M}%
_{m-1}(S_{1})$ and $\varepsilon>0$ as in conditions \textbf{L}$(p,m,\Pi)$ and
\textbf{UK}$(p,m,\Pi)$, we can use integration by parts to obtain%
\begin{align*}
\left(  \xi_{i}-\xi_{i}^{\tau,m}\right)  \hat{I}_{\beta}(\hat{\psi}_{i,\beta
}^{\tau,m}(\cdot))_{0,t} &  =\int_{0}^{t}\Theta_{s}^{g^{i},\tau,m}dY_{s}%
^{i}\int_{0}^{t}\hat{I}_{_{-}\beta}(\hat{\psi}_{i,\beta}^{\tau,m}%
(s,\cdot))_{0,s}dV_{s}^{\beta_{1}}\\
&  =\int_{0}^{t}\left(  \int_{0}^{s}\Theta_{u}^{g^{i},\tau,m}dY_{u}%
^{i}\right)  \hat{I}_{_{-}\beta}(\hat{\psi}_{i,\beta}^{\tau,m}(s,\cdot
))_{0,s}dV_{s}^{\beta_{1}}\\
&  +\int_{0}^{t}\hat{I}_{\beta}(\hat{\psi}_{i,\beta}^{\tau,m}(\cdot
))_{0,s}\Theta_{s}^{g^{i},\tau,m}dY_{s}^{i}.
\end{align*}
Note that
\begin{align*}
&  \mathbb{\tilde{E}}\left[  \left(  \int_{0}^{t}\left(  \int_{0}^{s}%
\Theta_{u}^{g^{i},\tau,m}dY_{u}^{i}\right)  \hat{I}_{_{-}\beta}(\hat{\psi
}_{i,\beta}^{\tau,m}(s,\cdot))_{0,s}dV_{s}^{\beta_{1}}\right)  ^{2}\right]  \\
&  =\int_{0}^{t}\mathbb{\tilde{E}}\left[  \left(  \left(  \int_{0}^{s}%
\Theta_{u}^{g^{i},\tau,m}dY_{u}^{i}\right)  \hat{I}_{_{-}\beta}(\hat{\psi
}_{i,\beta}^{\tau,m}(s,\cdot))_{0,s}\right)  ^{2}\right]  ds\\
&  \leq\int_{0}^{t}\mathbb{\tilde{E}}\left[  \left(  \int_{0}^{s}\Theta
_{u}^{g^{i},\tau,m}dY_{u}^{i}\right)  ^{\frac{2(2+\varepsilon)}{\varepsilon}%
}\right]  ^{\varepsilon/(2+\varepsilon)}\mathbb{\tilde{E}}\left[  \left(
\hat{I}_{_{-}\beta}(\hat{\psi}_{i,\beta}^{\tau,m}(s,\cdot))_{0,s}\right)
^{2+\varepsilon}\right]  ^{2/(2+\varepsilon)}ds\\
&  \leq\mathbb{\tilde{E}}\left[  \sup_{0\leq s\leq t}\left(  \int_{0}%
^{s}\Theta_{u}^{g^{i},\tau,m}dY_{u}^{i}\right)  ^{\frac{2(2+\varepsilon
)}{\varepsilon}}\right]  ^{\varepsilon/(2+\varepsilon)}\int_{0}^{t}%
\mathbb{\tilde{E}}\left[  \left(  \hat{I}_{_{-}\beta}(\hat{\psi}_{i,\beta
}^{\tau,m}(s,\cdot))_{0,s}\right)  ^{2+\varepsilon}\right]  ^{2/(2+\varepsilon
)}ds\\
&  \leq C\left(  \varepsilon,t\right)  \mathbb{\tilde{E}}\left[  \left(
\int_{0}^{t}|\Theta_{u}^{g^{i},\tau,m}|^{2}du\right)  ^{\frac{2+\varepsilon
}{\varepsilon}}\right]  ^{\varepsilon/(2+\varepsilon)}\int_{0}^{t}%
\mathbb{\tilde{E}}\left[  \left(  \hat{I}_{_{-}\beta}(\hat{\psi}_{i,\beta
}^{\tau,m}(s,\cdot))_{0,s}\right)  ^{2+\varepsilon}\right]  ^{2/(2+\varepsilon
)}ds\\
&  \leq C\left(  \varepsilon,t\right)  \mathbb{\tilde{E}}\left[  \sup_{0\leq
u\leq t}|\Theta_{u}^{g^{i},\tau,m}|^{\frac{2(2+\varepsilon)}{\varepsilon}%
}\right]  ^{\varepsilon/(2+\varepsilon)}\\
&  \times\left(  \sup_{0\leq s_{|\beta|}<\cdots<s_{1}\leq t}\mathbb{\tilde{E}%
}\left[  \left\vert \hat{\psi}_{i,\beta}^{\tau,m}\left(  s_{1},...,s_{|\beta
|}\right)  \right\vert ^{2+\varepsilon}\right]  \right)  ^{2/(2+\varepsilon
)}\\
&  <\infty,
\end{align*}
where we have used the Burkholder-Davis-Gundy inequality several times,
assumption \textbf{UK}$(p,m,\Pi)$ and that $\sup_{0\leq u\leq t}|\Theta
_{u}^{g^{i},\tau,m}|$ has moments of all orders. This yields that
\[
\int_{0}^{v}\left(  \int_{0}^{s}\Theta_{u}^{g^{i},\tau,m}dY_{u}^{i}\right)
\hat{I}_{_{-}\beta}(\hat{\psi}_{i,\beta}^{\tau,m}(s,\cdot))_{0,s}dV_{s}%
^{\beta_{1}},\quad v\in\lbrack0,t],
\]
is a $\mathcal{F}_{v}\vee\mathcal{Y}_{t}$-martingale and it vanishes when
taking conditional expectation with respect to $\mathcal{F}_{0}\vee
\mathcal{Y}_{t}.$ Therefore,%
\begin{align*}
&  \mathbb{\tilde{E}}\left[  \left(  \xi_{i}-\xi_{i}^{\tau,m}\right)  \hat
{I}_{\beta}(\hat{\psi}_{i,\beta}^{\tau,m}(\cdot))_{0,t}|\mathcal{Y}%
_{t}\right]  \\
&  =\int_{0}^{t}\mathbb{\tilde{E}}\left[  \hat{I}_{\beta}(\hat{\psi}_{i,\beta
}^{\tau,m}(\cdot))_{0,s}\Theta_{s}^{g^{i},\tau,m}|\mathcal{Y}_{t}\right]
dY_{s}^{i}\\
&  =\int_{0}^{t}\mathbb{\tilde{E}}\left[  \int_{0}^{s}\cdots\int
_{0}^{s_{|\beta|-1}}\hat{\psi}_{i,\beta}^{\tau,m}(s_{1},...,s_{|\beta
|})dV_{s_{|\beta|}}^{\beta_{|\beta|}}\cdots dV_{s_{1}}^{\beta_{1}}\right.  \\
&  \quad\left.  \times\left(  \sum_{\alpha\in\mathcal{M}_{m-1}^{R}(S_{0}%
)}I_{\alpha}(L^{\alpha}g^{i}(X_{\cdot}))_{\tau(s),s}\right)  |\mathcal{Y}%
_{t}\right]  dY_{s}^{i}\\
&  =\int_{0}^{t}\mathbb{\tilde{E}}\left[  \int_{0}^{s}\cdots\int
_{0}^{s_{|\beta|-1}}\hat{\psi}_{i,\beta}^{\tau,m}(s_{1},...,s_{|\beta
|})dV_{s_{|\beta|}}^{\beta_{|\beta|}}\cdots dV_{s_{1}}^{\beta_{1}}\right.  \\
&  \times\left.  \left(  \sum_{\alpha\in\mathcal{M}_{m-1}^{R}(S_{0})}\int
_{0}^{s}\cdots\int_{0}^{s_{2}}%
{\displaystyle\prod\limits_{j=1}^{m}}
\boldsymbol{1}_{\{s_{j}>\tau(s)\}}L^{\alpha}g^{i}(X_{s_{1}})dV_{s_{1}}%
^{\alpha_{1}}\cdots dV_{s_{m}}^{\alpha_{m}}\right)  |\mathcal{Y}_{t}\right]
dY_{s}^{i},
\end{align*}
where $dV_{s_{j}}^{\alpha_{j}}=ds_{j}$ if $\alpha_{j}=0.$ Taking conditional
expectation with respect to $\mathcal{F}_{0}\vee\mathcal{Y}_{t}$ we get that
the only term that does not vanish is the one corresponding to $\alpha
=\alpha(\beta)\triangleq\alpha_{0}^{m-|\beta|}\ast(\beta_{\left\vert
\beta\right\vert },...,\beta_{1}).$ Hence, defining
\[
\Lambda_{\alpha(\beta)}(s_{\left\vert \beta\right\vert })\triangleq\int
_{\tau(s)}^{s_{|\beta|}}\int_{\tau(s)}^{s_{|\beta|+1}}\cdots\int_{\tau
(s)}^{s_{m-1}}L^{\alpha(\beta)}g^{i}(X_{s_{m}})ds_{m}\cdots ds_{|\beta|+1},
\]
we get that
\begin{align*}
&  \mathbb{\tilde{E}}\left[  \Lambda_{\alpha(\beta)}(s_{\left\vert
\beta\right\vert })^{2}|\mathcal{Y}_{t}\right]  \\
&  \leq\delta^{m-\left\vert \beta\right\vert }\int_{\tau(s)}^{s_{|\beta|}}%
\int_{\tau(s)}^{s_{|\beta|+1}}\cdots\int_{\tau(s)}^{s_{m-1}}\mathbb{\tilde{E}%
}\left[  |L^{\alpha(\beta)}g^{i}(X_{s_{m}})|^{2}|\mathcal{Y}_{t}\right]
ds_{m}\cdots ds_{|\beta|+1}\\
&  \leq\delta^{m-\left\vert \beta\right\vert }\int_{\tau(s)}^{s_{|\beta|}}%
\int_{\tau(s)}^{s_{|\beta|+1}}\cdots\int_{\tau(s)}^{s_{m-1}}\mathbb{\tilde{E}%
}\left[  (1+|X_{\tau(s)}|^{r})^{2}\right]  ds_{m}\cdots ds_{|\beta|+1}\\
&  \leq C\delta^{m-\left\vert \beta\right\vert }\int_{\tau(s)}^{s_{|\beta|}%
}\int_{\tau(s)}^{s_{|\beta|+1}}\cdots\int_{\tau(s)}^{s_{m-1}}ds_{m}\cdots
ds_{|\beta|+1},
\end{align*}
and we can write%
\begin{align*}
&  \mathbb{\tilde{E}}\left[  \hat{I}_{\beta}(\hat{\psi}_{i,\beta}^{\tau
,m}(\cdot))_{0,s}\Theta_{s}^{g^{i},\tau,m}|\mathcal{Y}_{t}\right]  ^{2}\\
&  =\mathbb{\tilde{E}}\left[  I_{\alpha(\beta)}(L^{\alpha(\beta)}%
g^{i}(X_{\cdot}))_{\tau(s),s}\hat{I}_{\beta}(\hat{\psi}_{i,\beta}^{\tau
,m}(\cdot))_{\tau(s),s}|\mathcal{Y}_{t}\right]  ^{2}\\
&  =\mathbb{\tilde{E}}\left[  \int_{\tau(s)}^{s}\cdots\int_{\tau
(s)}^{s_{|\beta|-1}}\Lambda_{\alpha(\beta)}(s_{\left\vert \beta\right\vert
})\hat{\psi}_{i,\beta}^{\tau,m}(s_{1},...,s_{\left\vert \beta\right\vert
})ds_{|\beta|}\cdots ds_{1}|\mathcal{Y}_{t}\right]  ^{2}\\
&  \leq\delta^{|\beta|}\int_{\tau(s)}^{s}\cdots\int_{\tau(s)}^{s_{|\beta|-1}%
}\mathbb{\tilde{E}}\left[  \Lambda_{\alpha(\beta)}(s_{\left\vert
\beta\right\vert })\hat{\psi}_{i,\beta}^{\tau,m}(s_{1},...,s_{\left\vert
\beta\right\vert })|\mathcal{Y}_{t}\right]  ^{2}ds_{|\beta|}\cdots ds_{1}\\
&  \leq\delta^{|\beta|}\int_{\tau(s)}^{s}\cdots\int_{\tau(s)}^{s_{|\beta|-1}%
}\mathbb{\tilde{E}}\left[  \Lambda_{\alpha(\beta)}(s_{\left\vert
\beta\right\vert })^{2}|\mathcal{Y}_{t}\right]  \mathbb{\tilde{E}}\left[
\hat{\psi}_{i,\beta}^{\tau,m}(s_{1},...,s_{\left\vert \beta\right\vert }%
)^{2}|\mathcal{Y}_{t}\right]  ds_{|\beta|}\cdots ds_{1}\\
&  \leq C\delta^{m}\int_{\tau(s)}^{s}\cdots\int_{\tau(s)}^{s_{|\beta|-1}}%
\int_{\tau(s)}^{s_{|\beta|}}\cdots\int_{\tau(s)}^{s_{m-1}}ds_{m}\cdots
ds_{|\beta|+1}\\
&  \quad\times\mathbb{\tilde{E}}\left[  \hat{\psi}_{i,\beta}^{\tau,m}%
(s_{1},...,s_{\left\vert \beta\right\vert })^{2}|\mathcal{Y}_{t}\right]
ds_{|\beta|}\cdots ds_{1}\\
&  =C\delta^{m}\int_{\tau(s)}^{s}\cdots\int_{\tau(s)}^{s_{m-1}}\mathbb{\tilde
{E}}\left[  \hat{\psi}_{i,\beta}^{\tau,m}(s_{1},...,s_{\left\vert
\beta\right\vert })^{2}|\mathcal{Y}_{t}\right]  ds_{m}\cdots ds_{1}.
\end{align*}
Finally,
\begin{align*}
&  \mathbb{\tilde{E}}\left[  \mathbb{\tilde{E}}\left[  \left(  \xi_{i}-\xi
_{i}^{\tau,m}\right)  \hat{I}_{\beta}(\hat{\psi}_{i,\beta}^{\tau,m}%
(\cdot))_{0,t}|\mathcal{Y}_{t}\right]  ^{2p}\right]  \\
&  =\mathbb{\tilde{E}}\left[  \left(  \int_{0}^{t}\mathbb{\tilde{E}}[\hat
{I}_{\beta}(\hat{\psi}_{i,\beta}^{\tau,m}(\cdot))_{0,s}\Theta_{s}^{g^{i}%
,\tau,m}|\mathcal{Y}_{t}]dY_{s}^{i}\right)  ^{2p}\right]  \\
&  \leq\mathbb{\tilde{E}}\left[  \left(  \int_{0}^{t}\mathbb{\tilde{E}}%
[\hat{I}_{\beta}(\hat{\psi}_{i,\beta}^{\tau,m}(\cdot))_{0,s}\Theta_{s}%
^{g^{i},\tau,m}|\mathcal{Y}_{t}]^{2}ds\right)  ^{p}\right]  \\
&  \leq C(p)\delta^{pm}\mathbb{\tilde{E}}\left[  \left(  \int_{0}^{t}%
\int_{\tau(s)}^{s}\cdots\int_{\tau(s)}^{s_{m-1}}\mathbb{\tilde{E}}\left[
\left\vert \hat{\psi}_{i,\beta}^{\tau,m}(s_{1},...,s_{\left\vert
\beta\right\vert })\right\vert ^{2}|\mathcal{Y}_{t}\right]  ds_{m}\cdots
ds_{1}ds\right)  ^{p}\right]  \\
&  \leq C(p,t)\delta^{2pm-1}\int_{0}^{t}\int_{\tau(s)}^{s}\cdots\int_{\tau
(s)}^{s_{m-1}}\mathbb{\tilde{E}}\left[  \left\vert \hat{\psi}_{i,\beta}%
^{\tau,m}(s_{1},...,s_{\left\vert \beta\right\vert })\right\vert ^{2p}\right]
ds_{m}\cdots ds_{1}ds\\
&  \leq C(p,t)\delta^{2pm}\sup_{\tau\in\Pi}\sup_{0\leq s_{|\beta|}%
<\cdots<s_{1}\leq t}\mathbb{\tilde{E}}\left[  \left\vert \hat{\psi}_{i,\beta
}^{\tau,m}(s_{1},...,s_{|\beta|})\right\vert ^{2p}\right]  \\
&  \leq C\delta^{2pm}.
\end{align*}

\underline{\textbf{Term }$A_{i,3}$}$:$

For any $i=1,...,d_{Y},\beta=(\beta_{1},...,\beta_{m})\in\mathcal{M}_{m-1}%
^{R}(S_{1})$, we can use integration by parts to obtain%
\begin{align*}
\left(  \xi_{i}-\xi_{i}^{\tau,m}\right)  \hat{I}_{\beta}(\psi_{i,\beta}%
^{\tau,m}(\cdot))_{0,t}  &  =\int_{0}^{t}\Theta_{s}^{g^{i},\tau,m}dY_{s}%
^{i}\int_{0}^{t}\hat{I}_{\beta_{-}}(\psi_{i,\beta}^{\tau,m}(\cdot
,s))_{0,s}dV_{s}^{\beta_{1}}\\
&  =\int_{0}^{t}\left(  \int_{0}^{s}\Theta_{u}^{g^{i},\tau,m}dY_{u}%
^{i}\right)  \hat{I}_{\beta_{-}}(\psi_{i,\beta}^{\tau,m}(\cdot,s))_{0,s}%
dV_{s}^{\beta_{1}}\\
&  +\int_{0}^{t}\hat{I}_{\beta}(\psi_{i,\beta}^{\tau,m}(\cdot))_{0,s}%
\Theta_{s}^{g^{i},\tau,m}dY_{s}^{i}.
\end{align*}
As for the term $A_{i,2},$ one can show that
\[
\int_{0}^{v}\left(  \int_{0}^{s}\Theta_{u}^{g^{i},\tau,m}dY_{u}^{i}\right)
\hat{I}_{\beta_{-}}(\psi_{i,\beta}^{\tau,m}(\cdot,s))_{0,s}dV_{s}^{\beta_{1}%
},\quad v\in\lbrack0,t],
\]
is a $\mathcal{F}_{v}\vee\mathcal{Y}_{t}$-martingale and it vanishes when
taking conditional expectation with respect to $\mathcal{F}_{0}\vee
\mathcal{Y}_{t}.$ Therefore,%
\begin{align*}
&  \mathbb{\tilde{E}}\left[  \left(  \xi_{i}-\xi_{i}^{\tau,m}\right)  \hat
{I}_{\beta}(\psi_{i,\beta}^{\tau,m}(\cdot))_{0,t}|\mathcal{Y}_{t}\right] \\
&  =\int_{0}^{t}\mathbb{\tilde{E}}\left[  \hat{I}_{\beta}(\psi_{i,\beta}%
^{\tau,m}(\cdot))_{0,s}\Theta_{s}^{g^{i},\tau,m}|\mathcal{Y}_{t}\right]
dY_{s}^{i}\\
&  =\int_{0}^{t}\mathbb{\tilde{E}}\left[  \int_{0}^{s}\cdots\int_{0}^{s_{m-1}%
}\psi_{i,\beta}^{\tau,m}(s_{1},...,s_{m})dV_{s_{1}}^{\beta_{1}}\cdots
dV_{s_{m}}^{\beta_{m}}\right. \\
&  \times\left.  \left(  \sum_{\alpha\in\mathcal{M}_{m-1}^{R}(S_{0})}%
I_{\alpha}(L^{\alpha}g^{i}(X_{\cdot}))_{\tau(s),s}\right)  |\mathcal{Y}%
_{t}\right]  dY_{s}^{i}\\
&  =\int_{0}^{t}\mathbb{\tilde{E}}\left[  \int_{0}^{s}\cdots\int_{0}^{s_{m-1}%
}\psi_{i,\beta}^{\tau,m}(s_{1},...,s_{m})dV_{s_{1}}^{\beta_{1}}\cdots
dV_{s_{m}}^{\beta_{m}}\right. \\
&  \times\left.  \left(  \sum_{\alpha\in\mathcal{M}_{m-1}^{R}(S_{0})}\int
_{0}^{s}\cdots\int_{0}^{s_{2}}%
{\displaystyle\prod\limits_{j=1}^{m}}
\boldsymbol{1}_{\{s_{j}>\tau(s)\}}L^{\alpha}g^{i}(X_{s_{1}})dV_{s_{1}}%
^{\alpha_{1}}\cdots dV_{s_{m}}^{\alpha_{m}}\right)  |\mathcal{Y}_{t}\right]
dY_{s}^{i},
\end{align*}
where $dV_{s_{j}}^{\alpha_{j}}=ds_{j}$ if $\alpha_{j}=0.$ Taking conditional
expectation with respect to $\mathcal{F}_{0}\vee\mathcal{Y}_{t},$ the only
term that does not vanish is the one corresponding to $\alpha=\alpha
(\beta)\triangleq(\beta_{m},...,\beta_{1})$ and we get that%
\begin{align*}
&  \mathbb{\tilde{E}}\left[  \hat{I}_{\beta}(\psi_{i,\beta}^{\tau,m}%
(\cdot))_{0,s}\Theta_{s}^{g^{i},\tau,m}|\mathcal{Y}_{t}\right]  ^{2}\\
&  =\mathbb{\tilde{E}}\left[  \int_{\tau(s)}^{s}\cdots\int_{\tau(s)}^{s_{m-1}%
}\mathbb{\tilde{E}}\left[  \psi_{i,\beta}^{\tau,m}(s_{1}...,s_{m}%
)L^{\alpha(\beta)}g^{i}(X_{s_{m}})|\mathcal{F}_{0}\vee\mathcal{Y}_{t}\right]
ds_{m}\cdots ds_{1}|\mathcal{Y}_{t}\right]  ^{2}\\
&  \leq\delta^{m}\mathbb{\tilde{E}}\left[  \int_{\tau(s)}^{s}\cdots\int
_{\tau(s)}^{s_{m-1}}\mathbb{\tilde{E}}\left[  \psi_{i,\beta}^{\tau,m}%
(s_{1}...,s_{m})^{2}|\mathcal{F}_{0}\vee\mathcal{Y}_{t}\right]  \mathbb{\tilde
{E}}\left[  |L^{\beta}g^{i}(X_{s_{1}})|^{2}|\mathcal{F}_{0}\vee\mathcal{Y}%
_{t}\right]  ds_{m}\cdots ds_{1}|\mathcal{Y}_{t}\right] \\
&  \leq C\delta^{m}\int_{\tau(s)}^{s}\cdots\int_{\tau(s)}^{s_{m-1}%
}\mathbb{\tilde{E}}\left[  (1+|X_{\tau(s)}|^{r})^{2}\right]  \mathbb{\tilde
{E}}\left[  \psi_{i,\beta}^{\tau,m}(s_{1}...,s_{m})^{2}|\mathcal{Y}%
_{t}\right]  ds_{m}\cdots ds_{1}\\
&  \leq C\delta^{m}\int_{\tau(s)}^{s}\cdots\int_{\tau(s)}^{s_{m-1}%
}\mathbb{\tilde{E}}\left[  \psi_{i,\beta}^{\tau,m}(s_{1}...,s_{m}%
)^{2}|\mathcal{Y}_{t}\right]  ds_{m}\cdots ds_{1}.
\end{align*}
Therefore,%
\begin{align*}
&  \mathbb{\tilde{E}}\left[  \mathbb{\tilde{E}}\left[  \left(  \xi_{i}-\xi
_{i}^{\tau,m}\right)  \hat{I}_{\beta}(\hat{\psi}_{i,\beta}^{\tau,m}%
(\cdot))_{0,t}|\mathcal{Y}_{t}\right]  ^{2p}\right] \\
&  =\mathbb{\tilde{E}}\left[  \left(  \int_{0}^{t}\mathbb{\tilde{E}}\left[
\hat{I}_{\beta}(\psi_{i,\beta}^{\tau,m}(\cdot))_{0,s}\Theta_{s}^{g^{i},\tau
,m}|\mathcal{Y}_{t}\right]  dY_{s}^{i}\right)  ^{2p}\right] \\
&  \leq\mathbb{\tilde{E}}\left[  \left(  \int_{0}^{t}\mathbb{\tilde{E}}%
[\hat{I}_{\beta}(\psi_{i,\beta}^{\tau,m}(\cdot))_{0,s}\Theta_{s}^{g^{i}%
,\tau,m}|\mathcal{Y}_{t}]^{2}ds\right)  ^{p}\right] \\
&  \leq C(p)\delta^{pm}\mathbb{\tilde{E}}\left[  \left(  \int_{0}^{t}%
\int_{\tau(s)}^{s}\cdots\int_{\tau(s)}^{s_{m-1}}\mathbb{\tilde{E}}%
[\psi_{i,\beta}^{\tau,m}(s_{1},...,s_{m})^{2}|\mathcal{Y}_{t}]ds_{m}\cdots
ds_{1}ds\right)  ^{p}\right] \\
&  \leq C(p,t)\delta^{2pm-1}\int_{0}^{t}\int_{\tau(s)}^{s}\cdots\int_{\tau
(s)}^{s_{m-1}}\mathbb{\tilde{E}}\left[  \left\vert \psi_{i,\beta}^{\tau
,m}(s_{1},...,s_{m})\right\vert ^{2p}\right]  ds_{m}\cdots ds_{1}ds\\
&  \leq C(p,t)\delta^{2pm}\sup_{\tau\in\Pi}\sup_{0\leq s_{m}<\cdots<s_{1}\leq
t}\mathbb{\tilde{E}}\left[  \left\vert \psi_{i,\beta}^{\tau,m}(s_{1}%
,...,s_{m})\right\vert ^{2p}\right] \\
&  \leq C\delta^{2pm}.
\end{align*}

\end{proof}

\section{Proof of Theorem \ref{Theo_Main_Filtering_2}\label{Sec_ProofTheo}}

We will deduce the result from Theorem \ref{Theo_Main}. Therefore, we have to
verify that conditions \textbf{L}$(p,2,\Pi(t,\delta_{0}))$\ and \textbf{UK}%
$(p,2,\Pi(t,\delta_{0}))$\textbf{\ }are satisfied for the particular setting
of the filtering problem, where $\delta_{0}$ is given by equation
(\ref{Equ_Delta_0})$.$ Note that in this case the function $\psi
:\mathbb{R}^{d_{Y}+1}\rightarrow\mathbb{R}$ is given by%
\[
\psi(z)=\mathbf{exp}(z)\triangleq\exp(\sum_{i=0}^{d_{Y}}z_{i}),\quad
z\in\mathbb{R}^{d_{Y}+1},
\]
and, for $i=0,...,d_{Y},$ we have
\[
\eta_{i}^{\tau,2}=\int_{0}^{1}\varphi(X_{t})\partial_{i}\mathbf{exp}%
(s\xi+(1-s)\xi^{\tau,2})ds=\int_{0}^{1}\varphi(X_{t})\mathbf{exp}%
(s\xi+(1-s)\xi^{\tau,2})ds,
\]
where $\xi$ and $\xi^{\tau,2}$ are computed with $g_{i}=h_{i},i=1,...,d_{Y}%
,g_{0}=-\frac{1}{2}(h_{1}^{2}+\cdots+h_{d_{Y}}^{2})$.

Before we can proceed we require some preliminary results. The next two lemmas
are needed to verify condition \textbf{L}$(p,2,\Pi(t,\delta_{0})).$

\begin{lemma}
\label{Lemma_Zt_p_integrability}Let $h\in\mathcal{B}_{b}$. Then, for any
$p\in\mathbb{R}$ one has
\[
\mathbb{\tilde{E}}\left[  Z_{t}^{p}\right]  =\mathbb{\tilde{E}}[\mathbf{exp}%
(p\xi)<\infty.
\]

\end{lemma}

\begin{proof}
We have that%
\begin{align*}
\mathbb{\tilde{E}}[\mathbf{exp}(p\xi)]  &  =\mathbb{\tilde{E}}\left[
\exp\left(  p\sum_{i=1}^{d_{Y}}\int_{0}^{t}h^{i}(X_{s})dY_{s}^{i}-\frac{p}%
{2}\sum_{i=1}^{d_{Y}}\int_{0}^{t}h^{i}\left(  X_{s}\right)  ^{2}ds\right)
\right] \\
&  =\mathbb{\tilde{E}}\left[  \mathcal{E}\left(  p\sum_{i=1}^{d_{Y}}\int
_{0}^{\cdot}h^{i}(X_{s})dY_{s}^{i}\right)  _{t}\exp\left(  \frac{p^{2}-p}%
{2}\sum_{i=1}^{d_{Y}}\int_{0}^{t}h^{i}\left(  X_{s}\right)  ^{2}ds\right)
\right] \\
&  \leq\exp\left(  \frac{p^{2}+|p|}{2}d_{Y}t\left\Vert h\right\Vert _{\infty
}^{2}\right)  \mathbb{\tilde{E}}\left[  \mathcal{E}\left(  p\sum_{i=1}^{d_{Y}%
}\int_{0}^{\cdot}h^{i}(X_{s})dY_{s}^{i}\right)  _{t}\right] \\
&  =\exp\left(  \frac{p^{2}+|p|}{2}d_{Y}t\left\Vert h\right\Vert _{\infty}%
^{2}\right)  <\infty,
\end{align*}
where $\mathcal{E}(p\sum_{i=1}^{d_{Y}}\int_{0}^{\cdot}h^{i}(X_{s})dY_{s}%
^{i})_{t}$ denotes the stochastic exponential, which is a (genuine) martingale
by Novikov's criterion with expectation equal to $1$.
\end{proof}

\begin{lemma}
\label{Lemma_XiTau2_p_integrability}Assume that $f,\sigma\in\mathcal{B}_{b}%
$\ and $h\in\mathcal{B}_{b}\cap C_{b}^{2}.$ Let $p\geq1$ be fixed and $\tau$
be a partition with mesh size
\[
\delta<\left(  p\left\Vert Lh\right\Vert _{\infty}\sqrt{d_{Y}d_{V}}\right)
^{-1},
\]
where
\[
\left\Vert Lh\right\Vert _{\infty}\triangleq\max_{\substack{i=1,...d_{Y}%
\\r=1,...,d_{V}}}\left\Vert L^{r}h^{i}\right\Vert _{\infty}.
\]
Then, one has that
\[
\mathbb{\tilde{E}}\left[  \mathbf{exp}(p\xi^{\tau,2})\right]  <\infty.
\]

\end{lemma}

The proof of Lemma \ref{Lemma_XiTau2_p_integrability} is quite technical and
is done in the last section. The next two lemmas are crucial to verify
Condition \textbf{UK}$(p,2,\Pi(t,\delta_{0})).$

\begin{lemma}
\label{LemmaMomentMallDer}If $X_{t}\in\mathbb{R}^{d_{X}}$ is the solution to
\[
X_{t}=x+\int_{0}^{t}A_{0}\left(  X_{s}\right)  ds+\sum_{j=1}^{d_{V}}\int
_{0}^{t}A_{j}\left(  X_{s}\right)  dV_{s}^{j},
\]
where $A_{0},A_{1},...,A_{d_{V}}$ are $N$-times continuously differentiable
with bounded derivatives of order greater or equal than one and $V_{t}%
=(V_{t}^{1},...,V_{t}^{d_{V}})$ is a $d_{V}$-dimensional Brownian motion.
Then, $X_{t}^{i}\in\mathbb{D}^{N,p},p\geq1,t\in\lbrack0,T],i=1,...,d_{X}.$
Furthermore, for any $p\geq1$ one has that
\[
\sup_{r_{1},r_{2},...,r_{k}\in\lbrack0,T]}\mathbb{E}\left[  \sup_{r_{1}\vee
r_{2}\vee\cdots\vee r_{k}\leq t\leq T}\left\vert D_{r_{1},r_{2},...,r_{k}%
}^{j_{1},j_{2},...,j_{k}}X_{t}^{i}\right\vert ^{p}\right]  <\infty,\quad1\leq
k\leq N.
\]

\begin{proof}
See Nualart \cite{Nu06}, Theorem 2.2.1. and 2.2.2.
\end{proof}
\end{lemma}

\begin{lemma}
\label{LemmaBoundedIteratedMD1}Assume $f,\sigma\in\mathcal{B}_{b}\cap
C_{b}^{2},h\in\mathcal{B}_{b}\cap C_{b}^{4}$ and $\varphi\in C_{P}^{2}.$ Then,
the random vector $\eta^{\tau,2}=(\eta_{i}^{\tau,2})_{i=0}^{d_{Y}}$ belongs to
$\mathbb{D}^{2,p},p\geq1.$ Moreover,%
\[
\sup_{r_{1},...,r_{\left\vert \alpha\right\vert }\in\lbrack0,t]}%
\mathbb{\tilde{E}}\left[  \left\vert D_{r_{1},...,r_{_{\left\vert
\alpha\right\vert }}}^{\alpha_{1},...,\alpha_{\left\vert \alpha\right\vert }%
}\eta_{i}^{\tau,2}\right\vert ^{p}\right]  <\infty,
\]
for any $i\in\{0,...,d_{Y}\}$ and $\alpha\in\mathcal{M}_{2}(S_{1}).$
\end{lemma}

The proof of Lemma \ref{LemmaBoundedIteratedMD1} is done in the last section.

\begin{remark}
The proof of Lemma \ref{LemmaBoundedIteratedMD1} can be adapted to the case of
$\psi\in C_{P}^{m+1}(\mathbb{R}^{d_{Y}+1};\mathbb{R)}$ without any requirement
on the partition mesh. Hence, if we assume that $\psi\in C_{P}^{m+1}%
,\varphi\in C_{P}^{m},$ $f,\sigma\in C_{b}^{m}$ and that condition
\textbf{G}$(m)$ holds, then conditions \textbf{L}$(p,m,\Pi(t))$\textbf{\ }and
\textbf{UK}$(p,m,\Pi(t))$ also hold and Theorem \ref{Theo_Main}\ can be applied.
\end{remark}

We are finally ready to put everything together and deduce Theorem
\ref{Theo_Main_Filtering_2}.

\begin{proof}
[Proof of Theorem \ref{Theo_Main_Filtering_2}]We will deduce the result from
Theorem \ref{Theo_Main}. Hence, we only need to check that conditions
\textbf{S}$(2)$,\textbf{G}$(2)$,\textbf{L}$(p,2,\Pi(t,\delta_{0}))$%
\textbf{\ }and \textbf{UK}$(p,2,\Pi(t,\delta_{0}))$\textbf{\ }are satisfied.
As $f,\sigma\in C_{b}^{2}$ and $X_{0}$ has moments of all orders, condition
\textbf{S}$(2)$ is satisfied. Moreover, as $g_{i}=h_{i},i=1,...,d_{Y}%
,g_{0}=-\frac{1}{2}(h_{1}^{2}+\cdots+h_{d_{Y}}^{2})$ and $h\in C_{P}^{4}$ we
also have that condition \textbf{G}$(2)$ is satisfied. By H\"{o}lder
inequality and inequality $\left(  \ref{Equ_Ineq_Exponentials}\right)  $ we
get that%
\begin{align*}
&  \mathbb{\tilde{E}}\left[  |\varphi(X_{t})\partial_{i}\psi(s\xi
+(1-s)\xi^{\tau,2})|^{2p+\varepsilon}\right]  \\
&  \leq\mathbb{\tilde{E}}\left[  \left\vert \varphi(X_{t})\right\vert
^{2p+\varepsilon}\mathbf{exp}((2p+\varepsilon)\{s\xi+(1-s)\xi^{\tau
,2}\})\right]  \\
&  \leq\mathbb{\tilde{E}}\left[  \left\vert \varphi(X_{t})\right\vert
^{\frac{(2p+\varepsilon)(2p+\varepsilon^{\prime})}{\varepsilon^{\prime
}-\varepsilon}}\right]  ^{(\varepsilon^{\prime}-\varepsilon)/(2p+\varepsilon
^{\prime})}\\
&  \quad\times\mathbb{\tilde{E}}\left[  \mathbf{exp}((2p+\varepsilon^{\prime
})\{s\xi+(1-s)\xi^{\tau,2}\})\right]  ^{(2p+\varepsilon)/(2p+\varepsilon
^{\prime})}\\
&  \leq C(p,\varepsilon,\varepsilon^{\prime})\left(  \mathbb{\tilde{E}}\left[
\mathbf{exp}\left(  (2p+\varepsilon^{\prime})\xi\right)  \right]
+\mathbb{\tilde{E}}\left[  \mathbf{exp}((2p+\varepsilon^{\prime})\xi^{\tau
,2})\right]  \right)  ^{(2p+\varepsilon)/(2p+\varepsilon^{\prime})},
\end{align*}
where $\varepsilon^{\prime}>\varepsilon>0$ are such that $\mathbb{\tilde{E}%
}\left[  \mathbf{exp}((2p+\varepsilon^{\prime})\xi^{\tau,2})\right]  <\infty,$
which exist due to Lemma \ref{Lemma_XiTau2_p_integrability} and the fact that
$\delta<\delta_{0.}.$ Note that we can apply Lemma
\ref{Lemma_XiTau2_p_integrability} because $f,\sigma\in\mathcal{B}_{b}$\ and
$h\in\mathcal{B}_{b}\cap C_{b}^{2}.$ Combining with Lemma
\ref{Lemma_Zt_p_integrability}\ we can conclude that condition \textbf{L}%
$(p,2,\Pi(t,\delta_{0}))$ holds. Moreover, condition \textbf{UK}%
$(p,2,\Pi(t,\delta_{0}))$ holds due to Lemma \ref{LemmaBoundedIteratedMD1} and
Theorem \ref{Theo_Stroock-Taylor}. Note that we can apply Lemma
\ref{LemmaBoundedIteratedMD1} because $f,\sigma\in\mathcal{B}_{b}\cap
C_{b}^{2},h\in\mathcal{B}_{b}\cap C_{b}^{4}$ and $\varphi\in C_{P}^{2}.$ Next,
applying Theorem \ref{Theo_Main} we get the desired rate of convergence for
the unnormalised conditional distribution $\rho_{t}^{\tau,2}$. To prove the
rate for the normalised conditional distribution observe that we can write%
\[
\pi_{t}^{\tau,2}\left(  \varphi\right)  -\pi_{t}\left(  \varphi\right)
=\frac{1}{\rho_{t}\left(  \boldsymbol{1}\right)  }\frac{\rho_{t}^{\tau
,2}\left(  \varphi\right)  }{\rho_{t}^{\tau,2}\left(  \boldsymbol{1}\right)
}\left(  \rho_{t}\left(  \boldsymbol{1}\right)  -\rho_{t}^{\tau,2}\left(
\boldsymbol{1}\right)  \right)  +\frac{1}{\rho_{t}\left(  \boldsymbol{1}%
\right)  }\left(  \rho_{t}^{\tau,2}\left(  \varphi\right)  -\rho_{t}\left(
\varphi\right)  \right)  ,
\]
hence%
\begin{align*}
&  \mathbb{E}\left[  \left\vert \pi_{t}\left(  \varphi\right)  -\pi_{t}%
^{\tau,2}\left(  \varphi\right)  \right\vert ^{p}\right]  \\
&  \leq C(p)\mathbb{\tilde{E}}\left[  \frac{Z_{t}}{\left\vert \rho_{t}\left(
\boldsymbol{1}\right)  \right\vert ^{p}}\left\{  \left\vert \pi_{t}^{\tau
,2}\left(  \varphi\right)  \right\vert ^{p}\left\vert \rho_{t}\left(
\boldsymbol{1}\right)  -\rho_{t}^{\tau,2}\left(  \boldsymbol{1}\right)
\right\vert ^{p}+\left\vert \rho_{t}^{\tau,2}\left(  \varphi\right)  -\rho
_{t}\left(  \varphi\right)  \right\vert ^{p}\right\}  \right]  \\
&  =C(p)\mathbb{\tilde{E}}\left[  \frac{\mathbb{\tilde{E}}[Z_{t}%
|\mathcal{Y}_{t}]}{\left\vert \rho_{t}\left(  \boldsymbol{1}\right)
\right\vert ^{p}}\left\{  \left\vert \pi_{t}^{\tau,2}\left(  \varphi\right)
\right\vert ^{p}\left\vert \rho_{t}\left(  \boldsymbol{1}\right)  -\rho
_{t}^{\tau,2}\left(  \boldsymbol{1}\right)  \right\vert ^{p}+\left\vert
\rho_{t}^{\tau,2}\left(  \varphi\right)  -\rho_{t}\left(  \varphi\right)
\right\vert ^{p}\right\}  \right]  \\
&  \leq C(p)\left\{  \mathbb{\tilde{E}}\left[  \left\vert \rho_{t}\left(
\boldsymbol{1}\right)  \right\vert ^{(1-p)}\left\vert \pi_{t}^{\tau,2}\left(
\varphi\right)  \right\vert ^{p}\left\vert \rho_{t}\left(  \boldsymbol{1}%
\right)  -\rho_{t}^{\tau,2}\left(  \boldsymbol{1}\right)  \right\vert
^{p}\right]  \right.  \\
&  \quad\left.  +\mathbb{\tilde{E}}\left[  \left\vert \rho_{t}\left(
\boldsymbol{1}\right)  \right\vert ^{(1-p)}\left\vert \rho_{t}^{\tau,2}\left(
\varphi\right)  -\rho_{t}\left(  \varphi\right)  \right\vert ^{p}\right]
\right\}  \\
&  \triangleq C(p)\left\{  A_{1}+A_{2}\right\}  .
\end{align*}
Applying H\"{o}lder inequality, we obtain%
\begin{align*}
A_{1} &  \leq\mathbb{\tilde{E}}\left[  \left\vert \rho_{t}\left(
\boldsymbol{1}\right)  \right\vert ^{2(1-p)}\left\vert \pi_{t}^{\tau,2}\left(
\varphi\right)  \right\vert ^{2p}\right]  ^{1/2}\mathbb{\tilde{E}}\left[
\left\vert \rho_{t}\left(  \boldsymbol{1}\right)  -\rho_{t}^{\tau,2}\left(
\boldsymbol{1}\right)  \right\vert ^{2p}\right]  ^{1/2}\\
&  \leq\mathbb{\tilde{E}}\left[  \left\vert \rho_{t}\left(  \boldsymbol{1}%
\right)  \right\vert ^{2(1-p)(2p+\varepsilon)/\varepsilon}\right]
^{\varepsilon/(2(2p+\varepsilon))}\mathbb{\tilde{E}}\left[  \left\vert \pi
_{t}^{\tau,2}\left(  \varphi\right)  \right\vert ^{2p+\varepsilon}\right]
^{2p/(2(2p+\varepsilon))}\\
&  \times\mathbb{\tilde{E}}\left[  \left\vert \rho_{t}\left(  \boldsymbol{1}%
\right)  -\rho_{t}^{\tau,2}\left(  \boldsymbol{1}\right)  \right\vert
^{2p}\right]  ^{1/2},
\end{align*}
and
\[
A_{2}\leq\mathbb{\tilde{E}}\left[  \left\vert \rho_{t}\left(  \boldsymbol{1}%
\right)  \right\vert ^{2(1-p)}\right]  ^{1/2}\mathbb{\tilde{E}}\left[
\left\vert \rho_{t}^{\tau,2}\left(  \varphi\right)  -\rho_{t}\left(
\varphi\right)  \right\vert ^{2p}\right]  ^{1/2}.
\]
Combining the bounds for the unnormalised distribution, the hypothesis on
$\pi_{t}^{\tau,2}\left(  \varphi\right)  $ and the fact that, due to Lemma
\ref{Lemma_Zt_p_integrability}, for any $q\leq0$ we have that
\[
\mathbb{\tilde{E}}\left[  \left\vert \rho_{t}\left(  \boldsymbol{1}\right)
\right\vert ^{q}\right]  =\mathbb{\tilde{E}}\left[  \left\vert \mathbb{\tilde
{E}}[Z_{t}|\mathcal{Y}_{t}]\right\vert ^{q}\right]  \leq\mathbb{\tilde{E}%
}\left[  Z_{t}^{q}\right]  <\infty,
\]
we can conclude.
\end{proof}

\begin{remark}
The assumption $\sup_{\tau\in\Pi(t,\delta_{0})}\mathbb{\tilde{E}}\left[
\left\vert \pi_{t}^{\tau,2}(\varphi)\right\vert ^{2p+\varepsilon}\right]
<\infty$ for some $\varepsilon>0$ is satisfied if $\varphi$ is bounded. If
$\varphi$ is unbounded, note that by using Jensen's inequality one has%
\begin{align*}
\mathbb{\tilde{E}}\left[  \left\vert \pi_{t}^{\tau,2}(\varphi)\right\vert
^{2p+\varepsilon}\right]   &  =\mathbb{\tilde{E}}\left[  \left\vert
\mathbb{\tilde{E}}\left[  \frac{\varphi(X_{t})Z_{t}^{\tau,2}}{\mathbb{\tilde
{E}}\left[  Z_{t}^{\tau,2}|\mathcal{Y}_{t}\right]  }|\mathcal{Y}_{t}\right]
\right\vert ^{2p+\varepsilon}\right] \\
&  \leq\mathbb{\tilde{E}}\left[  |\varphi(X_{t})|^{2p+\varepsilon}%
\mathbf{exp}((2p+\varepsilon)(\xi^{\tau,2}-\mathbb{\tilde{E}}[\xi^{\tau
,2}|\mathcal{Y}_{t}]))\right]  .
\end{align*}
Hence, if $\varphi$ has polynomial growth and $h\in\mathcal{B}_{b}\cap
C_{b}^{2}$\textbf{,} one can reason as in Lemma
\ref{Lemma_XiTau2_p_integrability} to obtain $\sup_{\tau\in\Pi(t,\delta_{0}%
)}\mathbb{\tilde{E}}\left[  \left\vert \pi_{t}^{\tau,2}(\varphi)\right\vert
^{2p+\varepsilon}\right]  <\infty.$
\end{remark}

\section{Proof of technical results\label{Sec_Technical}}

In this section we provide the proof for Lemmas
\ref{Lemma_XiTau2_p_integrability} and \ref{LemmaBoundedIteratedMD1}, which
are of a more technical nature.

\begin{proof}
[Proof of Lemma \ref{Lemma_XiTau2_p_integrability}]We can write $\mathbf{exp}%
\left(  p\xi^{\tau,2}\right)  \triangleq%
{\displaystyle\prod\limits_{i=1}^{4}}
\left(  K_{t}^{\tau,2,i}\right)  ^{p},$ where%
\begin{align*}
K_{t}^{\tau,2,1}  &  \triangleq\exp\left(  \sum_{i=1}^{d_{Y}}\sum_{r=1}%
^{d_{V}}\int_{0}^{t}L^{r}h^{i}(X_{\tau(s)})(V_{s}^{r}-V_{\tau(s)}^{r}%
)dY_{s}^{i}\right)  ,\\
K_{t}^{\tau,2,2}  &  \triangleq\exp\left(  \sum_{i=1}^{d_{Y}}\int_{0}%
^{t}\{h^{i}(X_{\tau(s)})+L^{0}h^{i}(X_{\tau(s)})(s-\tau(s))\}dY_{s}%
^{i}\right)  ,\\
K_{t}^{\tau,2,3}  &  \triangleq\exp\left(  -\frac{1}{2}\sum_{i=1}^{d_{Y}}%
\int_{0}^{t}\{(h^{i})^{2}\left(  X_{\tau(s)}\right)  +L^{0}((h^{i}%
)^{2})(X_{\tau(s)})(s-\tau(s))\}ds\right)  ,\\
K_{t}^{\tau,2,4}  &  \triangleq\exp\left(  -\frac{1}{2}\sum_{i=1}^{d_{Y}}%
\sum_{r=1}^{d_{V}}\int_{0}^{t}L^{r}((h^{i})^{2})(X_{\tau(s)})(V_{s}%
^{r}-V_{\tau(s)}^{r})\}ds\right)  .
\end{align*}
Let $\varepsilon>0,$ then, by H\"{o}lder inequality, we have
\[
\mathbb{\tilde{E}}\left[  \mathbf{exp}\left(  p\xi^{\tau,2}\right)  \right]
\leq\mathbb{\tilde{E}}\left[  \left\vert K_{t}^{\tau,2,1}\right\vert
^{p(1+\varepsilon)}\right]  ^{\frac{1}{1+\varepsilon}}\mathbb{\tilde{E}%
}\left[
{\displaystyle\prod\limits_{i=2}^{4}}
\left\vert K_{t}^{\tau,2,i}\right\vert ^{p\frac{(1+\varepsilon)}{\varepsilon}%
}\right]  ^{\frac{\varepsilon}{1+\varepsilon}}.
\]
Hence, the result follows by showing that $K_{t}^{\tau,2,1}$ has finite
$p(1+\varepsilon)$-moment and%
\begin{equation}
\mathbb{\tilde{E}}\left[
{\displaystyle\prod\limits_{i=2}^{4}}
\left\vert K_{t}^{\tau,2,i}\right\vert ^{p\frac{(1+\varepsilon)}{\varepsilon}%
}\right]  <\infty. \label{EquProductK}%
\end{equation}
Applying H\"{o}lder inequality twice, condition $\left(  \ref{EquProductK}%
\right)  $ follows by showing that $K_{t}^{\tau,2,i},i=2,...,4$ have finite
moments of all orders. In what follows, let $q\geq1$ be a fixed real constant.
We start by the easiest term, $K_{t}^{\tau,2,3}.$ We have that%
\[
\mathbb{\tilde{E}}\left[  \left\vert K_{t}^{\tau,2,3}\right\vert ^{q}\right]
\leq\exp\left(  \frac{qd_{Y}}{2}t(\left\Vert h\right\Vert _{\infty}^{2}%
+\delta\left\Vert L^{0}h^{2}\right\Vert _{\infty}\right)  <\infty,
\]
because $\left\Vert h\right\Vert _{\infty}^{2}$ and $\left\Vert L^{0}%
h^{2}\right\Vert _{\infty}=\max_{i=1,...,d_{Y}}\left\Vert L^{0}(h_{i}%
^{2})\right\Vert _{\infty}$ are finite due to the assumptions on $f,\sigma$
and $h.$ For the term $K_{t}^{\tau,2,4},$ we can write%
\begin{align*}
\mathbb{\tilde{E}}\left[  \left\vert K_{t}^{\tau,2,4}\right\vert ^{q}\right]
&  \leq\mathbb{\tilde{E}}\left[  \exp\left(  \frac{qd_{Y}d_{V}}{2}\left\Vert
L((h)^{2})\right\Vert _{\infty}\int_{0}^{t}\left\vert V_{s}^{1}-V_{\tau
(s)}^{1}\right\vert ds\right)  \right] \\
&  \leq\mathbb{\tilde{E}}\left[  \exp\left(  \frac{qd_{Y}d_{V}}{2}\left\Vert
L((h)^{2})\right\Vert _{\infty}t\sqrt{\delta}\left\vert V_{1}^{1}\right\vert
\right)  \right]  <\infty,
\end{align*}
because $\left\Vert L((h)^{2})\right\Vert _{\infty}=\max
_{\substack{i=1,...d_{Y}\\r=1,...,d_{V}}}\left\Vert L^{r}(h_{i}^{2}%
)\right\Vert _{\infty}$ is finite and $\left\vert V_{1}^{1}\right\vert $ has
exponential moments of any order.

For the term $K_{t}^{\tau,2,2},\ $we first condition with respect to
$\mathcal{F}_{t}^{V}=\sigma(V_{s},0\leq s\leq t)$ and use the fact that,
conditionally to $\mathcal{F}_{t}^{V}$, the stochastic integrals with respect
to $Y$ are Gaussian. We get
\begin{align*}
&  \mathbb{\tilde{E}}\left[  \left\vert K_{t}^{\tau,2,2}\right\vert
^{q}\right]  \\
&  =\mathbb{\tilde{E}}\left[  \mathbb{\tilde{E}}\left[  \exp\left(
q\sum_{i=1}^{d_{Y}}\int_{0}^{t}\{h^{i}(X_{\tau(s)})+L^{0}h^{i}(X_{\tau
(s)})(s-\tau(s))\}dY_{s}^{i}\right)  |\mathcal{F}_{t}^{V}\right]  \right]  \\
&  =\mathbb{\tilde{E}}\left[  \exp\left(  \frac{q^{2}}{2}\sum_{i=1}^{d_{Y}%
}\int_{0}^{t}\left\{  h^{i}(X_{\tau(s)})+L^{0}h^{i}(X_{\tau(s)})(s-\tau
(s))\right\}  ^{2}ds\right)  \right]  \\
&  =\exp(q^{2}d_{Y}t\{\left\Vert h\right\Vert _{\infty}^{2}+\left\Vert
L^{0}h\right\Vert _{\infty}^{2}\})<\infty.
\end{align*}

Finally, the term $K_{t}^{\tau,2,1}$ is more delicate because, in order to
show that has finite $(p+\varepsilon)$-moment, a relationship between the mesh
of the partition $\delta$ and $p+\varepsilon$ is needed. Proceeding as with
the term $K_{t}^{\tau,2,2},$ we obtain
\begin{align*}
&  \mathbb{\tilde{E}}\left[  \left\vert K_{t}^{\tau,2,1}\right\vert
^{p(1+\varepsilon)}\right]  \\
&  =\mathbb{\tilde{E}}\left[  \exp\left(  p(1+\varepsilon)\sum_{i=1}^{d_{Y}%
}\sum_{r=1}^{d_{V}}\int_{0}^{t}L^{r}h^{i}(X_{\tau(s)})(V_{s}^{r}-V_{\tau
(s)}^{r})dY_{s}^{i}\right)  \right]  \\
&  =\mathbb{\tilde{E}}\left[  \prod_{i=1}^{d_{Y}}\mathbb{\tilde{E}}\left[
\exp\left(  \int_{0}^{t}p(1+\varepsilon)\sum_{r=1}^{d_{V}}L^{r}h^{i}%
(X_{\tau(s)})(V_{s}^{r}-V_{\tau(s)}^{r})dY_{s}^{i}\right)  |\mathcal{F}%
_{t}^{V}\right]  \right]  .
\end{align*}
Now, conditionally to $\mathcal{F}_{t}^{V},$ the terms in the exponential are
centered Gaussian random variables and we get that%
\begin{align*}
&  \mathbb{\tilde{E}}\left[  \left\vert K_{t}^{\tau,2,1}\right\vert
^{p(1+\varepsilon)}\right]  \\
&  =\mathbb{\tilde{E}}\left[  \prod_{i=1}^{d_{Y}}\exp\left(  \frac
{p^{2}(1+\varepsilon)^{2}}{2}\int_{0}^{t}\left(  \sum_{r=1}^{d_{V}}L^{r}%
h^{i}(X_{\tau(s)})(V_{s}^{r}-V_{\tau(s)}^{r})\right)  ^{2}ds\right)  \right]
\\
&  \leq\mathbb{\tilde{E}}\left[  \prod_{i=1}^{d_{Y}}\exp\left(  \frac
{p^{2}(1+\varepsilon)^{2}d_{V}}{2}\int_{0}^{t}\left(  \sum_{r=1}^{d_{V}}%
|L^{r}h^{i}(X_{\tau(s)})|^{2}(V_{s}^{r}-V_{\tau(s)}^{r})^{2}\right)
ds\right)  \right]  \\
&  =\mathbb{\tilde{E}}\left[  \exp\left(  \frac{p^{2}(1+\varepsilon)^{2}%
d_{Y}d_{V}\left\Vert Lh\right\Vert _{\infty}^{2}}{2}\sum_{r=1}^{d_{V}}\int
_{0}^{t}(V_{s}^{r}-V_{\tau(s)}^{r})^{2}ds\right)  \right]  \\
&  =\mathbb{\tilde{E}}\left[  \exp\left(  \frac{p^{2}(1+\varepsilon)^{2}%
d_{Y}d_{V}\left\Vert Lh\right\Vert _{\infty}^{2}}{2}\int_{0}^{t}(V_{s}%
^{1}-V_{\tau(s)}^{1})^{2}ds\right)  \right]  ^{d_{V}}.
\end{align*}
So we need to find conditions on $\beta>0,$ such that $\mathbb{\tilde{E}%
}\left[  \exp\left(  \beta\int_{0}^{t}(V_{s}^{1}-V_{\tau(s)}^{1}%
)^{2}ds\right)  \right]  <\infty.$ We can write
\begin{align*}
\mathbb{\tilde{E}}\left[  \exp\left(  \beta\int_{0}^{t}(V_{s}^{1}-V_{\tau
(s)}^{1})^{2}ds\right)  \right]   &  =\mathbb{\tilde{E}}\left[  \exp\left(
\beta\sum_{j=1}^{n}\int_{t_{j-1}}^{t_{j}}(V_{s}^{1}-V_{t_{j-1}}^{1}%
)^{2}ds\right)  \right]  \\
&  =\prod_{j=1}^{n}\mathbb{\tilde{E}}\left[  \exp\left(  \beta\int_{t_{j-1}%
}^{t_{j}}(V_{s}^{1}-V_{t_{j-1}}^{1})^{2}ds\right)  \right]  \\
&  \triangleq\prod_{j=1}^{n}\Theta\left(  \beta,\delta_{j}\right)  .
\end{align*}
Denote by $M_{t}\triangleq\sup_{0\leq s\leq t}V_{s}^{1}$ and recall that the
density of $M_{t}$ is given by $f_{M_{t}}\left(  x\right)  =\frac{2}%
{\sqrt{2\pi t}}e^{-\frac{x^{2}}{2t}}\boldsymbol{1}_{(0,\infty)}.$ Moreover,
note that for any $A>0,$%
\[
\frac{2}{\sqrt{2\pi\sigma^{2}}}\int_{0}^{\infty}\exp\left\{  -A\frac{x^{2}%
}{2\sigma^{2}}\right\}  dx=A^{-1/2}.
\]
Then, we have that
\begin{align*}
\Theta\left(  \beta,\delta_{j}\right)   &  \leq\mathbb{\tilde{E}}[\exp
(\beta\delta_{j}M_{\delta}^{2}]=\int_{0}^{\infty}\frac{2}{\sqrt{2\pi\delta
_{j}}}\exp\left\{  \beta\delta_{j}x^{2}-\frac{x^{2}}{2\delta_{j}}\right\}  \\
&  =\int_{0}^{\infty}\frac{2}{\sqrt{2\pi\delta_{j}}}\exp\left\{  -\left(
1-2\beta\delta_{j}^{2}\right)  \frac{x^{2}}{2\delta_{j}}\right\}  =\left(
1-2\beta\delta_{j}^{2}\right)  ^{-1/2}<\infty,
\end{align*}
as long as $1-2\beta\delta_{j}^{2}>0.$ On the other hand,%
\begin{align*}
\left(  1-2\beta\delta_{j}^{2}\right)  ^{-1} &  =\sum_{k=0}^{\infty}\left(
2\beta\delta_{j}^{2}\right)  ^{k}=1+2\beta\delta_{j}^{2}\left(  \sum
_{k=0}^{\infty}\left(  2\beta\delta_{j}^{2}\right)  ^{k}\right)  \\
&  \leq1+2\beta\delta_{j}^{2}\left(  \sum_{k=0}^{\infty}\left(  2\beta
\delta^{2}\right)  ^{k}\right)  =1+\frac{2\beta\delta_{j}^{2}}{1-2\beta
\delta^{2}}\\
&  \leq\exp\left(  \frac{2\beta\delta_{j}^{2}}{1-2\beta\delta^{2}}\right)  ,
\end{align*}
and, therefore,%
\begin{align*}
\prod_{j=1}^{n}\Theta\left(  \beta,\delta_{j}\right)    & \leq\prod_{j=1}%
^{n}\exp\left(  \frac{\beta\delta_{j}^{2}}{1-2\beta\delta^{2}}\right)
\leq\exp\left(  \frac{\beta\sum_{j=1}^{n}\delta_{j}^{2}}{1-2\beta\delta^{2}%
}\right)  \\
& \leq\exp\left(  \frac{\beta\delta t}{1-2\beta\delta^{2}}\right)  <\infty.
\end{align*}

As $\beta=\frac{p^{2}(1+\varepsilon)^{2}d_{Y}d_{V}\left\Vert Lh\right\Vert
_{\infty}^{2}}{2}$ and $\varepsilon>0$ can be made arbitrary small we get the
following condition for the partition mesh $\delta<\left(  p\left\Vert
Lh\right\Vert _{\infty}\sqrt{d_{Y}d_{V}}\right)  ^{-1}.$
\end{proof}

\begin{proof}
[Proof of Lemma \ref{LemmaBoundedIteratedMD1}]To ease the notation we are just
going to give the proof for $d_{V}=d_{Y}=d_{X}=1.$ Let $F\triangleq
\varphi(X_{t})\ $and $G\triangleq\int_{0}^{1}\mathbf{exp}(s\xi+(1-s)\bar{\xi
}^{\tau,2})ds.$ Then, by Leibniz's rule, for any $\alpha\in\mathcal{M}%
_{2}(S_{1})$ we get%
\begin{align*}
D_{r_{1},...,r_{_{\left\vert \alpha\right\vert }}}^{\alpha_{1},...,\alpha
_{\left\vert \alpha\right\vert }}\eta_{i}^{\tau,2}  & =D_{r_{1}%
,...,r_{_{\left\vert \alpha\right\vert }}}^{\left\vert \alpha\right\vert }%
\eta_{i}^{\tau,2}=D_{r_{1},...,r_{_{\left\vert \alpha\right\vert }}%
}^{\left\vert \alpha\right\vert }\left(  FG\right)  \\
& =\sum_{k=0}^{\left\vert \alpha\right\vert }\binom{\left\vert \alpha
\right\vert }{k}\left(  D_{r_{1},...,r_{k}}^{k}F\right)  (D_{r_{1}%
,...,r_{\left\vert \alpha\right\vert -k}}^{\left\vert \alpha\right\vert -k}G),
\end{align*}
and applying H\"{o}lder's inequality one has that
\begin{align*}
&  \mathbb{\tilde{E}}\left[  \left\vert D_{r_{1},...,r_{_{\left\vert
\alpha\right\vert }}}^{\left\vert \alpha\right\vert }\eta_{i}^{\tau
,2}\right\vert ^{p}\right]  \\
\quad &  \leq C\left(  \left\vert \alpha\right\vert \right)  \sum
_{k=0}^{\left\vert \alpha\right\vert }\binom{\left\vert \alpha\right\vert }%
{k}\mathbb{\tilde{E}}\left[  \left\vert \left(  D_{r_{1},...,r_{k}}%
^{k}F\right)  (D_{r_{1},...,r_{\left\vert \alpha\right\vert -k}}^{\left\vert
\alpha\right\vert -k}G)\right\vert ^{p}\right]  \\
&  \leq C\left(  \left\vert \alpha\right\vert \right)  \sum_{k=0}^{\left\vert
\alpha\right\vert }\binom{\left\vert \alpha\right\vert }{k}\mathbb{\tilde{E}%
}\left[  \left\vert D_{r_{1},...,r_{k}}^{k}F\right\vert ^{p\frac
{(1+\varepsilon)}{\varepsilon}}\right]  ^{\varepsilon/(1+\varepsilon)}\\
&  \times\mathbb{\tilde{E}}\left[  \left\vert D_{r_{1},...,r_{\left\vert
\alpha\right\vert -k}}^{\left\vert \alpha\right\vert -k}G\right\vert
^{p(1+\varepsilon)}\right]  ^{1/(1+\varepsilon)},
\end{align*}
for some $\varepsilon>0$. Hence, the result follows if we show that
\begin{align}
\sup_{r_{1},...,r_{k}\in\lbrack0,t]}\mathbb{\tilde{E}}\left[  |D_{r_{1}%
,...,r_{k}}^{k}F|^{q}\right]   &  <\infty,\quad0\leq k\leq\left\vert
\alpha\right\vert ,\label{EquMDF}\\
\sup_{r_{1},...,r_{k}\in\lbrack0,t]}\mathbb{\tilde{E}}[|D_{r_{1},...,r_{k}%
}^{k}G|^{p(1+\varepsilon)}] &  <\infty,\quad0\leq k\leq\left\vert
\alpha\right\vert ,\label{EquMDG}%
\end{align}
for any $q\geq1$ and some $\varepsilon>0.$

\underline{\textit{Proof of} $\left(  \ref{EquMDF}\right)  $}$:$

If $k=0,$ using that $f\in C_{b}^{2},\sigma\in C_{b}^{2}$ and $\varphi\in
C_{P}^{2}$\textbf{,} we have that $\mathbb{\tilde{E}}[|F|^{q}]=\mathbb{\tilde
{E}}[|\varphi(X_{t})|^{q}]<\infty.$ If $1\leq k\leq\left\vert \alpha
\right\vert ,$ we use Fa\`{a} di Bruno's formula to obtain an expression for
$D_{r_{1},...,r_{k}}^{k}F$ in terms of the so called partial Bell polynomials,
which are given by
\[
B_{k,a}(x_{1},...,x_{k})=\sum_{(j_{1},...,j_{k})\in\Lambda(k,a)}\frac
{k!}{j_{1}!\left(  1!\right)  ^{j_{1}}j_{2}!\left(  2!\right)  ^{j_{2}}\cdots
j_{k}!(k!)^{j_{k}}}x_{1}^{j_{1}}x_{2}^{j_{2}}\cdots x_{k}^{j_{k}},
\]
where $1\leq a\leq k$ and
\[
\Lambda(k,a)=\{(j_{1},...,j_{k})\in\mathbb{Z}_{+}^{k}:j_{1}+2j_{2}%
+\cdots+kj_{k}=k,j_{1}+j_{2}+\cdots+j_{k}=a\}.
\]
In particular, we have that%
\[
D_{r_{1},...,r_{k}}^{k}\varphi(X_{t})=\sum_{a=1}^{k}\varphi^{(a)}%
(X_{t})B_{k,a}(D_{r_{1}}^{1}X_{t},D_{r_{1},r_{2}}^{2}X_{t},...,D_{r_{1}%
,...,r_{k}}^{k}X_{t}).
\]
Hence, for any $q\geq1,$ applying Cauchy-Schwarz inequality we get
\begin{align*}
\mathbb{\tilde{E}}[|D_{r_{1},...,r_{k}}^{k}F|^{q}] &  \leq C(q,k)\sum
_{a=1}^{k}\mathbb{\tilde{E}}[|\varphi^{(a)}(X_{t})B_{k,a}(D_{r_{1}}^{1}%
X_{t},D_{r_{1},r_{2}}^{2}X_{t},...,D_{r_{1},...,r_{k}}^{k}X_{t})|^{q}]\\
&  \leq C(q,k)\sum_{a=1}^{k}\mathbb{\tilde{E}}[|\varphi^{(a)}(X_{t}%
)|^{2q}]^{1/2}\\
&  \quad\times\mathbb{\tilde{E}}[|B_{k,a}(D_{r_{1}}^{1}X_{t},D_{r_{1},r_{2}%
}^{2}X_{t},...,D_{r_{1},...,r_{k}}^{k}X_{t})|^{2q}]^{1/2}.
\end{align*}
The terms $\mathbb{\tilde{E}}[\left\vert \varphi^{(a)}(X_{t})\right\vert
^{2q}]<\infty,a=1,...,k,$ due to Remark \ref{RemarkM} and that $\varphi\in
C_{P}^{2}.$ On the other hand, using a generalized version of H\"{o}lder's
inequality we can bound
\[
\mathbb{\tilde{E}}[|B_{k,a}(D_{r_{1}}^{1}X_{t},D_{r_{1},r_{2}}^{2}%
X_{t},...,D_{r_{1},...,r_{k}}^{k}X_{t})|^{2q}],\quad1\leq a\leq k,
\]
by a sum of products of expectations of powers of Malliavin derivatives of
different orders. Combining this bound with Lemma \ref{LemmaMomentMallDer} we
get that the integrability condition $\left(  \ref{EquMDF}\right)  $ is satisfied.

\underline{\textit{Proof of }$\left(  \ref{EquMDG}\right)  $}$:$

First note that, by the convexity of the exponential function, we have that
\begin{align}
\mathbf{exp}(q\{s\xi+(1-s)\xi^{\tau,2}\})  &  =\exp\left(  sq\sum_{i=0}%
^{d_{Y}}\xi_{i}+(1-s)q\sum_{i=0}^{d_{Y}}\xi_{i}^{\tau,2}\right) \nonumber\\
&  \leq s\exp\left(  q\sum_{i=0}^{d_{Y}}\xi_{i}\right)  +(1-s)\exp\left(
q\sum_{i=0}^{d_{Y}}\xi_{i}^{\tau,2}\right) \nonumber\\
&  \leq\mathbf{exp}\left(  q\xi\right)  +\mathbf{exp}(q\xi_{i}^{\tau,2}),
\label{Equ_Ineq_Exponentials}%
\end{align}
where $p>0$ and $0\leq s\leq1.$

If $k=0,$ we have that
\begin{align*}
\mathbb{\tilde{E}}\left[  |G|^{p(1+\varepsilon)}\right]   &  =\mathbb{\tilde
{E}}\left[  \left\vert \int_{0}^{1}\mathbf{exp}(s\xi+(1-s)\xi^{\tau
,2})ds\right\vert ^{p(1+\varepsilon)}\right]  \\
&  \leq\int_{0}^{1}\mathbb{\tilde{E}}[\mathbf{exp}(p(1+\varepsilon
)(s\xi+(1-s)\xi^{\tau,2}))]ds\\
&  \leq\mathbb{\tilde{E}}[\mathbf{exp}(p(1+\varepsilon)\xi)]+\mathbb{\tilde
{E}}[\mathbf{exp}(p(1+\varepsilon)\xi^{\tau,2})]<\infty
\end{align*}
where we have used $\left(  \ref{Equ_Ineq_Exponentials}\right)  $ and Lemmas
\ref{Lemma_Zt_p_integrability} and \ref{Lemma_XiTau2_p_integrability}. If
$1\leq k\leq\left\vert \alpha\right\vert ,$ using the basic properties of the
Mallavin derivative and the definition of $\mathbf{exp}$, we have that
\begin{align}
D_{r_{1},...,r_{k}}^{k}G &  =\int_{0}^{1}D_{r_{1},...,r_{k}}^{k}%
\mathbf{exp}\left(  s\xi+(1-s)\xi^{\tau,2}\right)  ds,\nonumber\\
&  =\int_{0}^{1}D_{r_{1},...,r_{k}}^{k}\exp\left(  \sum_{i=0}^{d_{Y}}s\xi
_{i}+(1-s)\xi_{i}^{\tau,2}\right)  ds\label{EquDkG}\\
&  =\int_{0}^{1}D_{r_{1},...,r_{k}}^{k}\exp(\Theta_{s})ds,\nonumber
\end{align}
where $\Theta_{s}\triangleq\sum_{i=0}^{d_{Y}}s\xi_{i}+(1-s)\xi_{i}^{\tau,2}.$
Using again Fa\`{a} di Bruno's formula we get
\begin{align}
D_{r_{1},...,r_{k}}^{k}\exp(\Theta_{s}) &  =\sum_{a=1}^{k}\frac{d^{a}}{dx^{a}%
}\exp(\Theta_{s})B_{k,a}(D_{r_{1}}^{1}\Theta_{s},D_{r_{1},r_{2}}^{2}\Theta
_{s},...,D_{r_{1},...,r_{k}}^{k}\Theta_{s})\nonumber\\
&  =\exp(\Theta_{s})\sum_{a=1}^{k}B_{k,a}(D_{r_{1}}^{1}\Theta_{s}%
,D_{r_{1},r_{2}}^{2}\Theta_{s},...,D_{r_{1},...,r_{k}}^{k}\Theta
_{s}).\label{EquDkTheta}%
\end{align}
and, on the other hand,
\begin{align}
&  \left\vert B_{k,a}(D_{r_{1}}^{1}\Theta_{s},D_{r_{1},r_{2}}^{2}\Theta
_{s},...,D_{r_{1},...,r_{k}}^{k}\Theta_{s})\right\vert ^{p(1+\varepsilon
)}\nonumber\\
&  \leq C(p,\varepsilon,k,a)\sum_{(j_{1},...,j_{k})\in\Lambda(k,a)}\left(
\frac{k!}{j_{1}!\left(  1!\right)  ^{j_{1}}\cdots j_{k}!(k!)^{j_{k}}}\right)
^{p(1+\varepsilon)}\nonumber\\
&  \times\left\vert D_{r_{1}}^{1}\Theta_{s}\right\vert ^{p(1+\varepsilon
)j_{1}}\cdots\left\vert D_{r_{1},...,r_{k}}^{k}\Theta_{s}\right\vert
^{p(1+\varepsilon)j_{k}}\nonumber\\
&  \triangleq\tilde{B}_{k,a}^{p,\varepsilon}(\left\vert D_{r_{1}}^{1}%
\Theta_{s}\right\vert ,\left\vert D_{r_{1},r_{2}}^{2}\Theta_{s}\right\vert
,...,\left\vert D_{r_{1},...,r_{k}}^{k}\Theta_{s}\right\vert
).\label{EquModifiedBellPoly}%
\end{align}
Note also that by the linearity of the Malliavin derivative we have that
\begin{align}
\left\vert D_{r_{1},...,r_{a}}^{a}\Theta_{s}\right\vert  &  =\left\vert
D_{r_{1},...,r_{a}}^{a}\left(  \sum_{i=0}^{d_{Y}}s\xi_{i}+(1-s)\xi_{i}%
^{\tau,2}\right)  \right\vert \nonumber\\
&  =\left\vert \sum_{i=0}^{d_{Y}}sD_{r_{1},...,r_{a}}^{a}\xi_{i}%
+(1-s)D_{r_{1},...,r_{a}}^{a}\xi_{i}^{\tau,2}\right\vert \nonumber\\
&  \leq\sum_{i=0}^{d_{Y}}\left\vert D_{r_{1},...,r_{a}}^{a}\xi_{i}\right\vert
+\left\vert D_{r_{1},...,r_{a}}^{a}\xi_{i}^{\tau,2}\right\vert
,\label{Equ_LinearMalliavin}%
\end{align}
$1\leq a\leq k.$ Hence, combining equations $\left(  \ref{EquDkG}\right)
,\left(  \ref{EquDkTheta}\right)  ,\left(  \ref{EquModifiedBellPoly}\right)
,\left(  \ref{Equ_LinearMalliavin}\right)  $ and using Cauchy-Schwarz
inequality we get%
\begin{align*}
&  \mathbb{\tilde{E}}[|D_{r_{1},...,r_{k}}^{k}G|^{p(1+\varepsilon)}]\\
&  \leq\int_{0}^{1}\mathbb{\tilde{E}}[|D_{r_{1},...,r_{k}}^{k}\exp(\Theta
_{s})|^{p(1+\varepsilon)}]ds\\
&  \leq\mathbb{\tilde{E}[}\mathbf{exp}\left(  p(1+\varepsilon)\xi\right)
\Phi(\xi,\xi^{\tau,2})]+\mathbb{\tilde{E}[}\mathbf{exp}\left(  p(1+\varepsilon
)\xi^{\tau,2}\right)  \Phi(\xi^{\tau,2})]\\
&  \leq\left\{  \mathbb{\tilde{E}[}\mathbf{exp}\left(  p(1+\varepsilon
^{\prime})\xi\right)  ]^{(1+\varepsilon)/(1+\varepsilon^{\prime}%
)}+\mathbb{\tilde{E}[}\mathbf{exp}\left(  p(1+\varepsilon^{\prime})\xi
^{\tau,2}\right)  ]^{(1+\varepsilon)/(1+\varepsilon^{\prime})}\right\}  \\
&  \times\mathbb{\tilde{E}}\left[  \left\vert \Phi(\xi,\xi^{\tau
,2})\right\vert ^{\frac{1+\varepsilon^{\prime}}{\varepsilon^{\prime
}-\varepsilon}}\right]  ^{(\varepsilon^{\prime}-\varepsilon)/(1+\varepsilon
^{\prime})},
\end{align*}
where $\varepsilon^{\prime}>\varepsilon$ and
\begin{align*}
\Phi(\xi,\xi^{\tau,2}) &  \triangleq C(k,p)\sum_{a=1}^{k}C(p,k,a)\sum
_{(j_{1},...,j_{k})\in\Lambda(k,a)}\left(  \frac{k!}{j_{1}!\left(  1!\right)
^{j_{1}}\cdots j_{k}!(k!)^{j_{k}}}\right)  ^{p}\\
&  \times\left\vert \sum_{i=0}^{d_{Y}}\left\vert D_{r_{1}}^{1}\xi
_{i}\right\vert +\left\vert D_{r_{1}}^{1}\xi_{i}^{\tau,2}\right\vert
\right\vert ^{pj_{1}}\cdots\left\vert \sum_{i=0}^{d_{Y}}\left\vert
D_{r_{1},...,r_{k}}^{k}\xi_{i}\right\vert +\left\vert D_{r_{1},...,r_{k}}%
^{k}\xi_{i}^{\tau,2}\right\vert \right\vert ^{pj_{k}}.
\end{align*}
The integrability of $\mathbf{exp}\left(  p(1+\varepsilon^{\prime})\xi\right)
$ and $\mathbf{exp}\left(  p(1+\varepsilon^{\prime})\xi^{\tau,2}\right)  $
follows from Lemmas \ref{Lemma_Zt_p_integrability} and
\ref{Lemma_XiTau2_p_integrability}, respectively. By the particular form of
$\Phi(\xi,\xi^{\tau,2}),$ it is clear that using H\"{o}lder inequality we can
show that $\left(  \ref{EquMDG}\right)  $ holds, provided that%
\begin{align}
\sup_{r_{1},...,r_{a}\in\lbrack0,t]}\mathbb{\tilde{E}}[|D_{r_{1},...,r_{a}%
}^{a}\xi_{i}|^{q}] &  <\infty,\quad1\leq a\leq k,0\leq i\leq d_{Y}%
\label{Equ_MDXi}\\
\sup_{r_{1},...,r_{a}\in\lbrack0,t]}\mathbb{\tilde{E}}[|D_{r_{1},...,r_{a}%
}^{a}\xi_{i}^{\tau,2}|^{q}] &  <\infty,\quad1\leq a\leq k,0\leq i\leq
d_{Y}\label{Equ_MDXiTau1m_bar}%
\end{align}
for any $q\geq1.$ We shall prove the case $i=1,$ the case $i=0$ being similar,
and we will drop the index $i$ in what follows. By Fa\`{a} di Bruno's formula
\begin{align*}
D_{r_{1},...,r_{a}}^{a}\xi &  =\int_{0}^{t}D_{r_{1},...,r_{a}}^{a}%
h(X_{s})dY_{s}\\
&  =\int_{0}^{t}\left(  \sum_{l=1}^{a}h^{(l)}(X_{s})B_{a,l}(D_{r_{1}}^{1}%
X_{s},D_{r_{1},r_{2}}^{2}X_{s},...,D_{r_{1},...,r_{a}}^{a}X_{s})\right)
dY_{s}.
\end{align*}
Hence, by Burkholder-Davis-Gundy inequality, we get
\begin{align*}
&  \mathbb{\tilde{E}}[|D_{r_{1},...,r_{a}}^{a}\xi|^{q}]\\
&  \leq C(a,q,t)\\
&  \quad\times\sum_{l=1}^{a}\int_{0}^{t}\mathbb{\tilde{E}}[\left\vert
h^{(l)}(X_{s})B_{a,l}(D_{r_{1}}^{1}X_{s},D_{r_{1},r_{2}}^{2}X_{s}%
,...,D_{r_{1},...,r_{a}}^{a}X_{s})\right\vert ^{q}]ds\\
&  \leq C(a,q,t)\\
&  \quad\times\left\Vert h\right\Vert _{\infty,2}^{q}\sum_{l=1}^{a}\int
_{0}^{t}\mathbb{\tilde{E}}[\left\vert B_{a,l}(D_{r_{1}}^{1}X_{s}%
,D_{r_{1},r_{2}}^{2}X_{s},...,D_{r_{1},...,r_{a}}^{a}X_{s})\right\vert
^{q}]ds,
\end{align*}
where
\[
\left\Vert h\right\Vert _{\infty,2}\triangleq\sum_{l=0}^{2}\left\Vert
h^{(l)}\right\Vert _{\infty}<\infty,
\]
because $h\in C_{b}^{2}.$ Therefore, using a generalized version of H\"{o}lder
inequality and Lemma \ref{LemmaMomentMallDer} we get $\left(  \ref{Equ_MDXi}%
\right)  .$

On the other hand, by Leibniz's rule and the Burkholder-Davis-Gundy
inequality, we get
\begin{align*}
&  \mathbb{\tilde{E}}[|D_{r_{1},...,r_{a}}^{a}\xi^{\tau,2}|^{q}]\\
&  =\mathbb{\tilde{E}}\left[  \left\vert \sum_{\beta\in\mathcal{M}_{1}(S_{0}%
)}\int_{0}^{t}D_{r_{1},...,r_{a}}^{a}\left(  L^{\beta}h(X_{\tau(s)})I_{\beta
}(\boldsymbol{1})_{\tau(s),s}\right)  dY_{s}\right\vert ^{q}\right]  \\
&  =\mathbb{\tilde{E}}\left[  \left\vert \int_{0}^{t}\sum_{\beta\in
\mathcal{M}_{1}(S_{0})}\sum_{l=0}^{a}\binom{a}{l}\left(  D_{r_{1},...,r_{l}%
}^{l}L^{\beta}h(X_{\tau(s)}\right)  (D_{r_{1},...,r_{a-l}}^{a-l}I_{\beta
}(\boldsymbol{1})_{\tau(s),s})dY_{s}\right\vert ^{q}\right]  \\
&  \leq C(m,q,a,t)\\
&  \quad\times\int_{0}^{t}\sum_{\beta\in\mathcal{M}_{1}(S_{0})}\sum_{l=0}%
^{a}\binom{a}{l}^{q}\mathbb{\tilde{E}}[\left\vert \left(  D_{r_{1},...,r_{l}%
}^{l}L^{\beta}h(X_{\tau(s)}\right)  (D_{r_{1},...,r_{a-l}}^{a-l}I_{\beta
}(\boldsymbol{1})_{\tau(s),s})\right\vert ^{q}]ds,
\end{align*}
and the proof is further reduced to show for any $\beta\in\mathcal{M}%
_{1}(S_{0})$ and that
\begin{align}
\sup_{r_{1},...,r_{l}\in\lbrack0,t]}\mathbb{\tilde{E}}\left[  |D_{r_{1}%
,...,r_{l}}^{l}L^{\beta}h(X_{\tau(s)})|^{q}\right]   &  <\infty,\quad0\leq
l\leq a,\label{Equ_MalliavinLAlpha}\\
\sup_{r_{1},...,r_{l}\in\lbrack0,t]}\mathbb{\tilde{E}}[\left\vert
D_{r_{1},...,r_{l}}^{l}I_{\beta}(\boldsymbol{1})_{\tau(s),s}\right\vert ^{q}]
&  <\infty,\quad0\leq l\leq a.\label{Equ_MalliavinIteratedInt}%
\end{align}
The proof of $\left(  \ref{Equ_MalliavinLAlpha}\right)  $ is similar to the
proof of $\left(  \ref{EquMDF}\right)  .$ The proof of $\left(
\ref{Equ_MalliavinIteratedInt}\right)  $ is based on the well known fact, see
Proposition 1.2.7 and exercise 1.2.5. in Nualart \cite{Nu06}, that
$D_{r_{1},...,r_{l}}^{l}I_{\beta}(\boldsymbol{1})_{\tau(s),s}$ can be
expressed as linear combinations of iterated integrals of order lower than
$l.$ Then, the result follows from Lemma 5.7.5. in Kloeden and Platen.
\end{proof}

\end{document}